\documentclass[11pt]{article}

\usepackage{amsmath}
\usepackage{amssymb}
\usepackage{amsthm}
\usepackage{cite}
\usepackage{graphicx}
\usepackage{color}
\usepackage{enumitem}
\usepackage{float}
\usepackage{multirow}
\usepackage[left=1in,right=1in,top=1in,bottom=1in]{geometry}
\usepackage{xspace}
\usepackage{hyperref}
\usepackage{todonotes}
\usepackage[capitalise]{cleveref}

\renewcommand{\paragraph}[1]{\vspace{6pt} \noindent \textbf{#1}\xspace}

\numberwithin{equation}{section}
\newtheorem{theorem}{Theorem}[section]

\newtheorem{corollary}[theorem]{Corollary}
\newtheorem{observation}[theorem]{Observation}
\newtheorem{lemma}[theorem]{Lemma}
\newtheorem{proposition}[theorem]{Proposition}
\newtheorem{claim}[theorem]{Claim}

  {\end{tabular}\par\medskip}

\newtheorem{conjecture}[theorem]{Conjecture}
\theoremstyle{definition}

\newtheorem{remark}[theorem]{Remark}
\newtheorem{definition}[theorem]{Definition}
\newtheorem{example}[theorem]{Example}

\newcommand{\GL}{\mathrm{GL}}
\newcommand{\F}{\mathbb{F}}

\newcommand{\Q}{\mathbb{Q}}
\newcommand{\C}{\mathbb{C}}
\newcommand{\R}{\mathbb{R}}
\newcommand{\N}{\mathbb{N}}

\newcommand{\rk}{\mathrm{rank}}

\newcommand{\nilp}{\mathrm{NilptInd}}
\newcommand{\nil}{\mathrm{NilInd}}
\newcommand{\mwl}{\mathrm{MaxWalkLen}}

\newcommand{\mnss}{\mathrm{MaxBdMatSize}}

\newcommand{\MaxBdMatOrder}{\mathrm{MaxBdMatOrd}}
\newcommand{\lsss}{\mathrm{MaxBdRankDim}}
\newcommand{\MaxBdRankOrder}{\mathrm{MaxBdRankOrd}}
\newcommand{\opsys}{\mathcal{F}}

\newcommand{\poly}{\mathrm{poly}}

\newcommand{\M}{\mathrm{M}}
\newcommand{\T}{\mathrm{T}}

\renewcommand{\S}{\mathrm{S}}

\newcommand{\tuple}[1]{\mathbf{#1}}
\newcommand{\tens}[1]{\mathtt{#1}}
\newcommand{\spa}[1]{\mathcal{#1}}

\newcommand{\cA}{\spa{A}}
\newcommand{\cB}{\spa{B}}
\newcommand{\cC}{\spa{C}}
\newcommand{\cD}{\spa{D}}
\newcommand{\cG}{\spa{G}}

\newcommand{\cS}{\spa{S}}
\newcommand{\cT}{\spa{T}}

\newcommand{\tA}{\tens{A}}
\newcommand{\tB}{\tens{B}}

\newcommand{\tC}{\tens{C}}
\newcommand{\tD}{\tens{D}}

\newcommand{\tS}{\tens{S}}

\newcommand{\vB}{\tuple{B}}

\newcommand{\mas}{\mathrm{MaxAcySize}}

\newcommand{\mai}{\mathrm{MaxIndAcyOrd}}

\newcommand{\MSNT}{SNT}
\newcommand{\lns}{\mathrm{MaxNilDim}}

\newcommand{\mrns}{\mathrm{MaxIndNilDim}}
\newcommand{\colspan}{\operatorname{colspan}}
\newcommand{\rowspan}{\operatorname{rowspan}}

\newcommand{\Nsc}{\mathrm{MaxNscSize}}
\newcommand{\Rdc}{\mathrm{MaxRdcDim}}
\newcommand{\Insc}{\mathrm{MaxIndNscOrd}}
\newcommand{\Irdc}{\mathrm{MaxIndRdcDim}}

\newcommand{\colset}{\operatorname{colset}}

\newcommand{\E}{\mathrm{E}}

\newcommand{\aut}{\operatorname{Aut}}
\newcommand{\Conj}{\operatorname{Conj}}
\newcommand{\Cong}{\operatorname{Cong}}
\newcommand{\Aut}{\operatorname{Aut}}

\DeclareMathOperator{\rank}{rank}

\DeclareMathOperator{\End}{End}
\DeclareMathOperator{\Bil}{Bil}

\newcommand{\too}%
{\xrightarrow{\text{\raisebox{-3pt}{$\sim$}}\,}}

\newcommand\smat[1]{\left[\begin{smallmatrix}#1\end{smallmatrix}\right]}

\hypersetup{
    pdftitle={Connections between graphs and matrix spaces},
    pdfauthor={Yinan Li, Youming Qiao, Avi Wigderson, Yuval Wigderson, Chuanqi Zhang},
    bookmarksnumbered=true,     
    bookmarksopen=true,         
    bookmarksopenlevel=1,       
    colorlinks=true,            
    pdfstartview=Fit,           
    pdfpagemode=UseOutlines,    
    pdfpagelayout=OneColumn,
    pdfstartview=FitH
}

\title{
Connections between graphs and matrix spaces
}

\author{
Yinan Li\thanks{Graduate School of Mathematics, Nagoya University, Japan ({\tt Yinan.Li@math.nagoya-u.co.jp}). Research supported by MEXT Quantum Leap Flagship Program (MEXT Q-LEAP) Grant Number JPMXS0120319794.}
\and
Youming Qiao\thanks{Centre for Quantum Software and Information, 
 University of Technology Sydney, Australia ({\tt Youming.Qiao@uts.edu.au}). 
 Research supported by Australian Research Council DP200100950.
 }
\and
Avi Wigderson\thanks{School of Mathematics, Institute for Advanced Study, 
Princeton, New Jersey 08540 ({\tt avi@ias.edu}). Research supported by NSF grant CCF-1900460.
}
\and
Yuval Wigderson\thanks{Department of Mathematics, Stanford University, Stanford, 
CA 94305 ({\tt yuvalwig@stanford.edu}). Research supported by NSF GRFP Grant DGE-1656518.}
\and
Chuanqi Zhang\thanks{Centre for Quantum Software and Information, 
 University of Technology Sydney, Australia ({\tt 
 Chuanqi.Zhang@student.uts.edu.au}). Research supported by Australian Research 
  Council DP200100950.}
}

\date{\today}

\begin{document}

\maketitle

\begin{abstract}
Given a bipartite graph $G$, the graphical matrix space $\cS_G$ consists of 
matrices whose non-zero entries can only be at those positions corresponding 
to edges in $G$. Tutte (\emph{J. London Math. Soc.}, 1947), Edmonds (\emph{J. Res. Nat. Bur. Standards Sect. B }, 1967) and Lov\'asz (\emph{FCT}, 1979) observed connections between perfect matchings in $G$ 
and full-rank matrices in $\cS_G$. Dieudonn\'e (\emph{Arch. Math.}, 1948) proved a 
tight upper bound on the dimensions of those matrix spaces containing only 
singular matrices. The starting point of this paper is a simultaneous 
generalization of these two classical results: we show that the largest dimension 
over subspaces of $\cS_G$ containing only singular matrices is equal to the 
maximum 
size over subgraphs of $G$ without perfect matchings, based on Meshulam's proof of 
Dieudonn\'e's result
(\emph{Quart. J. Math.}, 1985).

Starting from this result, we go on to establish more connections between properties of graphs and matrix spaces. For example, we establish connections between acyclicity and nilpotency,
between strong connectivity and irreducibility, and between isomorphism and conjugacy/congruence. 
For each connection, we study three types of correspondences, namely the basic 
correspondence, the inherited correspondence (for subgraphs and subspaces), and 
the induced correspondence (for induced subgraphs and restrictions). Some 
correspondences lead to intriguing generalizations of classical results, such as Dieudonn\'e's result mentioned above, and a celebrated theorem of Gerstenhaber regarding the largest dimension of nil matrix spaces (\emph{Amer. J. 
Math.}, 1958). 

Finally, we show some implications of our results to quantum information and 
present open problems in computational complexity motivated by 
these results. 
\end{abstract}
\newpage
\tableofcontents


\section{Introduction}
\subsection{Overview}\label{subsec:overview}
Let $\M(n,\F)$ denote the vector space of $n\times n$ square matrices over a field 
$\F$. A vector subspace $\cS$ of $\M(n,\F)$ is called a \emph{matrix space}. 
Matrix 
spaces are basic and fundamental mathematical objects, and they arise naturally in 
many different areas of mathematics, physics, and computer science. In algebraic 
geometry, they arise in close connection with certain sheaves on projective space 
\cite{EH88}. In topology, they arise naturally in connection to linearly 
independent vector fields on spheres, which led to the development of the Adams 
operations on topological $K$-theory \cite{Adams62,ALP65a}. In invariant theory, 
they were used by Dieudonn\'e \cite{Die48} to classify the symmetries of the 
determinant, recovering a result of Frobenius \cite{Frobenius}. Gerstenhaber 
\cite{Ger58} used matrix spaces to make progress on Albert's problem in the theory 
of non-associative algebras. 
In group theory, Baer observes that (alternating) matrix spaces are closely 
connected with $p$-groups of class $2$~\cite{Bae38}. Recently, 
 they arise in the study of completely positive maps and quantum expanders in quantum information theory~\cite{quasirandom}. 
In computational complexity theory, they underlie the \emph{polynomial identity testing} problem, a 
central challenge in derandomization and algebraic complexity \cite{KI04}. 
Finally, in 
discrete mathematics and theoretical computer science, matrix spaces can be used 
to study matchings in graphs, and this connection has important algorithmic 
consequences \cite{Tut47,Edmonds67,Lov79,Lov89}. We return to this example shortly.

One important way of thinking of matrix spaces is as \emph{symbolic matrices}. 
Namely, we may choose a basis for $\cS$ and represent a generic element of $\cS$ 
as a generic linear combination of the basis elements. By doing so, we construct a 
matrix whose entries are homogeneous linear forms in some variables, and we can 
recover $\cS$ by substituting all possible elements of $\F$ into these variables. 
We discuss symbolic matrices in more detail in~\cref{subsec:complexity}. 

In this paper, we focus on matrix spaces \emph{of restricted support}. Namely, we 
fix a set of positions $E \subseteq [n]^2$, and study properties of matrix spaces 
$\cS \leq \M(n,\F)$ satisfying that every matrix $M \in \cS$ is supported on 
$E$, namely that the $(i,j)$th entry of $M$ is equal to zero for all $(i,j) \notin E$. Crucially for our purposes, 
we can encode the restricted support information as a \emph{graph}. Indeed, the 
support set $E \subseteq [n]\times[n]$ is a set of ordered pairs, and this can 
naturally 
be viewed as either the arc set of a directed graph on the vertex set $[n]$ or as 
the edge set of a bipartite graph with vertex set $[n]\times[n]$. Formally, we can 
make 
the following definition.
\begin{definition}[Graphical matrix spaces]\label{def: graphical matrix space}
For $(i, j)\in [n]^2$, let $\E_{i, j}$ be the elementary matrix in $\M(n, \F)$ 
where the $(i, j)$th entry is $1$, and the remaining entries are $0$.
Suppose $G=(L\cup R, E)$ is a bipartite graph, where $L=R=[n]$, or $G=([n], E)$ is 
a directed graph. The \emph{graphical matrix space} $\cS_G$ (over $\F$) corresponding to $G$ the 
subspace of $\M(n,\F)$ spanned by $\{\E_{i,j}\mid (i,j) \in E\}$.
\end{definition}
Thus, we see that a matrix space supported on the edges of $G$ is the same as a subspace of $\cS_G$. Additionally, we may view the symbolic matrix defining some $\cS \leq \cS_G$ as simply a collection of homogeneous linear forms, each of which is associated with an edge of $G$.

At first glance, it is not clear why encoding the restricted support as a graph is helpful or meaningful. However, it turns out that many natural linear-algebraic properties of matrix spaces correspond directly to graph-theoretic properties of (bipartite or directed) graphs, and this connection is captured by the association of the matrix space $\cS_G$ to the graph $G$. Moreover, the study of matrix spaces with restricted support---that is, subspaces of $\cS_G$---closely mirrors the study of subgraphs of $G$. Many specific connections of this type were studied in the past, as we discuss below. In this paper, we initiate a systematic study of the connections between graphs and their associated graphical matrix spaces. In many instances, these explorations yield surprising generalizations and extensions of known results in the ``full-support'' setting.

To motivate our results, we begin with a connection mentioned above, namely the connection between perfect matchings in bipartite graphs and singularity of matrix spaces, beginning with foundational works of Tutte \cite{Tut47}, Edmonds \cite{Edmonds67}, and Lov\'asz \cite{Lov79}. Let $G$ be a bipartite graph with both parts of size $n$. The key observation is that a linear-algebraic property of the matrix space $\cS_G$ encodes a graph-theoretic property of $G$. Namely, $G$ has a perfect matching if and only if $\cS_G$ contains a non-singular matrix. Indeed, if $G$ has a perfect matching, then the matrix in $\cS_G$ which has a $1$ on every edge of this perfect matching and zeroes elsewhere, is non-singular. Conversely, if $\cS_G$ contains a non-singular matrix, then the expansion of its determinant has at least one non-zero summand, which corresponds to a perfect matching in $G$. Equivalently, if we encode the matrix space $\cS_G$ as a symbolic matrix, then $G$ has a perfect matching if and only if the determinant of this symbolic matrix is not the zero polynomial.

The first important consequence of this connection is algorithmic. Indeed, Lov\'asz \cite{Lov79} used this connection to find an RNC algorithm\footnote{Informally, an efficient, probabilistic parallel algorithm.} for the problem of determining whether a bipartite graph has a perfect matching. Additionally, inspired by this connection, Edmonds \cite{Edmonds67}
asked whether one could devise an efficient, \emph{deterministic} algorithm to determine if such an arbitrary matrix space contains a non-singular matrix. This is one of the most important problems in computational complexity, and we return to it in \cref{subsec:complexity}.

In this paper, we extend the basic connection between perfect matchings in graphs 
and non-singular matrices in matrix spaces in several ways. First, we develop a 
number of other correspondences, showing that a graph (or directed graph) has some 
graph-theoretic property if and only if its associated graphical matrix space has 
some linear-algebraic property. In most instances this connection is fairly 
straightforward to prove (like in the case above), but in some others it is quite 
involved and requires a number of different ideas. Secondly, we prove a number of 
\emph{dimension theorems}, which demonstrate that in some instances, these 
correspondences between graphs and their graphical matrix spaces extend to their 
subgraphs and subspaces, respectively, and are surprisingly deep. 
Third, we prove a number of \emph{order theorems}, which demonstrate that in some 
instances, these correspondences between graphs and their graphical matrix spaces 
extend to their induced subgraphs and induced subspaces, respectively.

Continuing the discussion above, here is an example of one of our dimension theorems, which is a special case of~\cref{thm:non-matchable_subgraph}.
\begin{theorem}\label{thm:basic-matching-dimension}
    Let $G$ be a bipartite graph with two parts of size $n$, and let $\cS_G \leq \M(n,\F)$ be the associated matrix space. The maximum number of edges of a subgraph of $G$ with no perfect matching equals the largest dimension of a subspace of $\cS_G$ containing only singular matrices.
\end{theorem}
Note that if $H$ is an $m$-edge subgraph of $G$ with no perfect matching, then $\cS_H \leq \cS_G$ is an $m$-dimensional matrix space supported on the edges of $G$, and every matrix in $\cS_H$ is singular. \cref{thm:basic-matching-dimension} says that such examples are extremal in the sense of dimension: if we wish to construct a singular matrix space $\cS$ supported on the edges of $G$, then the biggest $\cS$ we can take is ``axis-aligned'', i.e., of the form $\cS_H$ for some $H \subseteq G$ with no perfect matching.

\cref{thm:basic-matching-dimension} generalizes a famous theorem of Dieudonn\'e\footnote{Dieudonn\'e was interested in classifying the symmetries of the determinant, and used this result to find a new proof of Frobenius's \cite{Frobenius} characterization of these symmetries.} \cite{Die48}, who proved that if every matrix in $\cS \leq \M(n,\F)$ is singular, then $\dim \cS \leq n(n-1)$. This is the special case of~\cref{thm:basic-matching-dimension} in which $G=K_{n,n}$, as it is easy to check that the largest subgraph of $K_{n,n}$ with no perfect matching has $n(n-1)$ edges. In other words, \cref{thm:basic-matching-dimension} is simply the ``restricted support'' version of Dieudonn\'e's theorem: it determines the largest dimension of a singular matrix space with (any) restricted support, just as Dieudonn\'e determined the largest dimension of a singular matrix space (with no restriction on its support). Our proof of~\cref{thm:basic-matching-dimension} is based on Meshulam's \cite{Mes85} proof of Dieudonn\'e's theorem, and indeed~\cref{thm:basic-matching-dimension} is almost implicit in Meshulam's work.

We also present an example of one of our order theorems, which is a 
special case of~\cref{thm:non-matchable_induced_subgraph}. For this, we need the 
following notion of induced matrices. Let $B\in \M(n,\F)$ and $U,V\leq\F^n$  
be dimension-$s$ and dimension-$t$ subspaces of $\F^n$, respectively. 
Viewing $B$ as a bilinear form $\F^n\times\F^n\to\F$, the induced matrix 
$B[U,V]\in \M(s\times t,\F)$ with respect to $U,V\leq\F^n$ is obtained by 
restricting the first argument to $U$ and the second to $V$. 
The \emph{induced subspace} $\cS[U,V]$ of $\cS\leq \M(n,\F)$ with respect to $U,V\leq\F^n$ is 
the matrix space consisting of $B[U,V]$ for all $B\in\cS$. The \emph{order} of $\cS[U,V]$ is 
$\dim(U)+\dim(V)=s+t$.  
\begin{theorem}\label{thm:basic-matching-order}
    Let $G$ be a bipartite graph with two parts of size $n$, and let $\cS_G \leq \M(n,\F)$ be the associated matrix space. The maximum number of vertices of an induced subgraph of $G$ with no perfect matching equals the largest order of an induced subspace of $\cS_G$ containing only singular matrices. 
\end{theorem}

Most of our results fit into a similar framework to what was described above. First, 
we have what we call a \emph{basic correspondence} between a 
graph-theoretic property $P$ and a linear-algebraic property $Q$, which simply 
means that a graph $G$ has property $P$ if and only if its graphical matrix space 
$\cS_G$ has property $Q$. The dimension theorem boosts this to what we call an 
\emph{inherited correspondence}: this says that the maximum number of edges in a 
spanning subgraph $H \subseteq G$ with property $P$ is equal to the largest 
dimension of a matrix space $\cS \leq \cS_G$ with property $Q$. Similarly, the order theorem 
boosts the basic correspondence to what we call an \emph{induced correspondence}: in this case, the 
number of vertices in the maximum induced subgraph $H$ of $G$ with property $P$ 
equals the order of the maximum induced subspace of $\cS_G$ with property $Q$. 

One useful way of thinking about inherited correspondences like~\cref{thm:basic-matching-dimension} is in terms of the symmetries at play. The natural group action on the set of bipartite graphs with both vertex sets of size $n$ is that of the group $\S_n \times \S_n$, which acts by simultaneously permuting the vertices in each of the two parts. This corresponds to relabeling the vertices in each part, and it of course preserves basic graph-theoretic properties such as the existence of a perfect matching. Similarly, in this context, a natural action on the space $\M(n,\F)$ of $n\times n$ matrices is the left-right action of $\GL(n,\F) \times \GL(n,\F)$, where we multiply on the left and the right by two invertible matrices. If we view matrices in $\M(n,\F)$ as linear maps $\F^n \to \F^n$, then this action corresponds to simultaneously changing bases in the domain and codomain. This action preserves many important linear-algebraic properties of $n\times n$ matrices, such as singularity. Of course, $\S_n \times \S_n$ is a subgroup of $\GL(n,\F) \times \GL(n,\F)$, which simply represents the fact that permuting the rows and columns of a matrix is one way of changing bases on the domain and codomain.

However, $\GL(n,\F) \times \GL(n,\F)$ is a much larger and richer group than $\S_n 
\times \S_n$, and its action on $\M(n,\F)$ is much more complicated than simply 
permuting the rows and columns. In the same way, the lattice of subspaces of 
$\cS_G$ is much richer than the lattice of subgraphs of $G$. Nonetheless, 
\cref{thm:basic-matching-dimension}, as well as our other dimension theorems, say 
that in some instances, this additional structure adds no extra flexibility: the 
largest dimension of a subspace of $\cS_G$ with some property $Q$ equals the 
largest dimension of ``axis-aligned'' subspaces $\cS_G$ with property $Q$, even 
though the set of such subspaces is much poorer.

In the rest of this introduction, we discuss our results in further detail, and explain some more about their connections to other topics. In~\cref{subsec:main-results}, we state our main results; this is arranged into a collection of subsubsections, with each one discussing the correspondences between some graph-theoretic and linear-algebraic properties. In~\cref{subsec:symmetries}, we discuss the importance of underlying group actions to our results. In~\cref{subsec:complexity,subsec:quantum}, we discuss the connections to computational complexity and quantum information theory, respectively. In~\cref{subsec:further-connections}, we broaden our scope, and mention other results which are of the form we discuss, namely correspondences between graphs and matrix spaces. Finally, in~\cref{subsec:open}, we discuss future research directions and open problems.

\subsection{Main results}\label{subsec:main-results}
We now describe our results in more detail. We begin with three topics in which we have dimension theorems: matchings in bipartite graphs, cycles in directed graphs, and strong connectivity of directed graphs. We then discuss several instances where we can prove a basic correspondence but not an inherited correspondence (and in some instances, we can even prove that inherited correspondence cannot hold). We first introduce some basic notation that we will use (and have used).

For $n\in \N$, $[n]:=\{1, 2, \dots, n\}$. Let $\F^n$ be the linear space of length-$n$ \emph{row vectors} over $\F$. 
Let $\M(n, \F)$ (resp.\ $\M(m\times n,\F)$) be the linear space of $n\times n$ (resp.\ $m \times n$) matrices over a field $\F$. For a matrix $B\in\M(n,\F)$ and a vector $v\in\F^n$, \emph{$B$ acts on $v$ from the right: $vB$ is another row vector in $\F^n$.} For a subspace $U\leq \F^n$, let $B(U)=\{uB: u\in U\}\leq \F^n$.
Let $\cS=\langle 
B_1,\dots,B_m\rangle$ be the \emph{linear span} of matrices $B_1,\dots, B_m\in 
\M(n,\F)$. We call $\cS$ a \emph{matrix space}, denoted as $\cS\leq \M(n,\F)$. Let 
$\GL(n,\F)$ be the group of $n\times n$ invertible matrices over $\F$. For a 
matrix $T\in\GL(n,\F)$, let $T^{-1}$ be its inverse. For a matrix $B\in \M(n,\F)$, 
let $B^t$ be its transpose. If $\F=\mathbb{C}$, let $B^*$ be its conjugate transpose. 

A directed graph is $G=(V, E)$, where $V$ is the vertex set 
and $E\subseteq V\times V$ is the \emph{arc} set. We shall mostly work with 
directed graphs with vertex sets being $V=[n]$. 

A bipartite graph is $G=(L\cup R, E)$, where $L$ and $R$ are the left and right 
vertex sets, and $E\subseteq L\times R$ is the \emph{edge} set. 

Except where otherwise stated, $\F$ is an arbitrary field; except where explicitly specified, our results hold for all fields.

\subsubsection{Matchings in bipartite graphs and ranks of matrices}
Given a bipartite graph $G=([m] \cup [n],E)$ with vertex parts $[m]$ and $[n]$, we 
let $\cS_G = \langle \E_{i,j}\mid (i,j) \in E\rangle \leq \M(m \times n,\F)$ be 
the 
graphical matrix space corresponding to $G$. A \emph{matching} is a subset of edges where any two edges do not share common vertices. The \emph{matching number} is the size of a maximum matching, i.e., a matching containing the largest possible number of edges.
We begin by stating formally the basic correspondence which was discussed above, between perfect matchings and non-singular matrices. 
\begin{theorem}[\hspace{1sp}{\cite[Theorem 1]{Edmonds67}}]\label{thm:edm67}
  Let $G=([m]\cup [n], E)$ be a bipartite graph with $m\leq n$ and $\cS_G\leq\M(m\times n, \F)$ be the graphical matrix space associated with $G$. Then for each $r \in [m]$, the matching number of $G$ is at most $r$ if and only if the rank of every matrix in $\cS_G$ is at most $r$. 
\end{theorem}
Thus, in our terminology, \cref{thm:edm67} establishes a basic correspondence between the graph property $P_r$ of the matching number being at most $r$ and the linear-algebraic property $Q_r$ of all matrices having rank at most $r$. Our next result, the more general form of~\cref{thm:basic-matching-dimension}, is the corresponding inherited correspondence between $P_r$ and $Q_r$.
\begin{theorem}\label{thm:non-matchable_subgraph}
    Let $G=([m]\cup [n], E)$ be a bipartite graph with $m\leq n$ and $\cS_G\leq\M(m\times n, \F)$ be the graphical matrix space associated with $G$. Then for each $r \in [m]$, the maximum size over subgraphs of $G$ whose matching number is at most $r$ equals the largest dimension over subspaces of $\cS_G$ in which every matrix is of rank at most $r$.
\end{theorem}
This theorem is proved in~\cref{subsec:ranks-inherited}. 
The proof is based on Meshulam's proof of the Dieudonn\'e--Flanders--Meshulam 
theorem \cite{Mes85}, which gives a bound on the dimension of an $m \times n$ 
matrix space containing matrices of rank at most $r$. In fact, the 
Dieudonn\'e--Flanders--Meshulam theorem corresponds to taking the complete 
bipartite graph $G=K_{m,n}$ in~\cref{thm:non-matchable_subgraph} .

In addition to the basic and inherited correspondences above, we prove one further 
correspondence between matchings and ranks in the context of \emph{induced 
correspondences}. This is a more general version of~\cref{thm:basic-matching-order}.
We first formally define what we mean an induced matrix space here.
\begin{definition}\label{def: induced bipartite graph}
	Let $\cS\leq\M(m\times n, \mathbb{F})$ be a matrix space over the field 
	$\mathbb{F}$. For a subspace $L\leq\mathbb{F}^m$ (resp.\ $R\leq\mathbb{F}^n$) 
	of dimension $s$ (resp.\ $t$), let $T_L$ (resp.\ $T_R$) be an $s\times m$ (resp.\ 
	$t\times n$) matrix whose rows span $L$ (resp.\ $R$). The \emph{induced 
	subspace} of $\cS$ on $L$ and $R$ is defined as $\cS[L, R]:=\{T_LBT_R^t\mid 
	B\in \cS\}\leq \M(s\times t, \mathbb{F})$.\footnote{While different bases lead 
	to different subspaces of $\M(s\times t,\mathbb{F})$, these subspaces are 
	unique up to equivalence. More precisely, let $T_L'$ (resp.\ $T_R'$) be 
	another $s\times m$ (resp.\ $t\times n$) matrix whose rows span $L$ (resp.\ 
	$R$) and let $\cS[L, R]':=\{T_L'B(T_R')^t\mid B\in \cS\}\leq \M(s\times t, 
	\mathbb{F})$, we can find invertible matrices $A_L\in\GL(s,\F)$ and 
	$A_R\in\GL(t,\F)$ such that $A_LT_L=T_L'$ and $A_RT_R=T_R'$. Thus 
	$A_L\cS[L,R]A_R^t=\cS[L,R]'$.} The \emph{order} of $\cS[L, R]$ is $s+t$. 
\end{definition}
We discuss in more detail why this is a natural definition in~\cref{subsec:symmetries} and for 
now simply remark that it mimics the notion of an induced subgraph. The induced 
correspondence in this context is as follows, proved 
in~\cref{subsec:ranks-induced}.
\begin{theorem}\label{thm:non-matchable_induced_subgraph}
	Let $G=([m]\cup [n], E)$ be a bipartite graph with $m\leq n$ and 
	$\cS_G\leq\M(m\times n, \F)$ be the graphical matrix space associated with 
	$G$. Then for each $r \in [m]$, the maximum order over induced subgraphs of 
	$G$ whose matching number is at most $r$ equals the maximum order of an 
	induced subspace of $\cS_G$ in which every matrix is of rank at most $r$. 
\end{theorem}

\subsubsection{Cycles in directed graphs and nilpotent matrices}\label{subsubsec:nil}

Given a directed graph $G=([n],E)$ with vertex set $[n]$, we let the corresponding 
graphical matrix space be $\cS_G = \langle \E_{i,j}\mid (i,j) \in E\rangle \leq 
\M(n,\F)$. Note that the natural group action on the set of directed graphs is 
that of $\S_n$, which acts by permuting the vertices. In the linear-algebraic 
setting, we will be focusing on the conjugation action of $\GL(n,\F)$ on 
$\M(n,\F)$. If we identify $\M(n,\F)$ with the set of endomorphisms of $\F^n$, 
then this action simply corresponds to changing the basis of $\F^n$. We have that 
$\S_n$ is a subgroup of $\GL(n,\F)$, corresponding to the fact that we may 
simultaneously permute the rows and columns of a matrix by a change of basis.

Recall that a directed graph is called \emph{acyclic} if it contains no directed 
cycle. We also say that a matrix space $S \leq \M(n,\F)$ is \emph{nil}\footnote{We 
adopt the terminology \emph{nil} here following the practice in algebra, where the 
distinction of nil and nilpotent algebras naturally leads to the definitions of 
nil and nilpotent matrix spaces; see~\cref{sec:nil}.} 
if every matrix in $S$ is nilpotent. We prove a number of correspondences between 
the graph-theoretic property of being acyclic and the linear-algebraic property of 
being nil, beginning with the following basic correspondence.
\begin{theorem}[Simplified version of~\cref{thm:acyclicity}]\label{thm:nil}
Let $G=([n], E)$ be a directed graph and $\cS_G\leq\M(n, \F)$ be the graphical 
matrix space associated with $G$. Then $G$ is acyclic if and only if $\cS_G$ is 
nil.
\end{theorem}
We are also able to prove more refined basic correspondences which generalize~\cref{thm:nil}. First, we prove a basic correspondence between the 
maximum path length in $G$ 
and the maximum nilpotent index of a matrix in $\cS_G$; see~\cref{thm:acyclicity}. Second, we prove a basic correspondence between the maximum cycle cover size 
in $G$ and the maximum number of zero eigenvalues of a matrix in $\cS_G$; see~\cref{thm:eigenval}.

Moreover, we prove a dimension version of~\cref{thm:nil}, i.e., to boost the basic correspondence to an inherited correspondence, as stated in the following theorem.
\begin{theorem}\label{thm:acyclic_subgraph}
Let $G=([n], E)$ be a directed graph and $\cS_G\leq\M(n, \F)$ be the graphical 
matrix space associated with $G$. The maximum number of arcs in an acyclic subgraph of $G$ equals the largest dimension of a nil subspace of $\cS_G$.
\end{theorem}
See~\cref{subsec:acyclic-inherited} for its proof.
\cref{thm:acyclic_subgraph} is a generalization of a well-known theorem of Gerstenhaber
about the largest dimension of a nil matrix spaces~\cite{Ger58} (cf.~\cref{remark: Gersten}), and its proof is adapted from 
de Seguins Pazzis's proof of Gerstenhaber's theorem~\cite{Paz13}.

We also prove an induced correspondence between the properties of being acyclic 
and nil.
Here, we need to reconcile the difference between conjugation (which preserves 
nilpotency) and congruence (where restrictions are natural), so the following 
alternative definition of induced subspaces is needed. This idea originates from 
the theory of non-commutative graphs~\cite{Wea21}. 
\begin{definition}\label{def:induced}
Let $\cS\leq\M(n, \mathbb{C})$ be a matrix space 
over the complex field $\mathbb{C}$. For a subspace $U\leq\mathbb{C}^n$ of dimension 
$d$, let $T_U$ be an $d\times n$ matrix whose rows form an orthonormal basis of 
$U$. The \emph{induced subspace} of $\cS$ on $U$ is defined as $\cS[U]:=\{ T_UB 
T_U^* \mid B\in \cS\}\leq \M(d, \mathbb{C})$.\footnote{While different orthonormal bases 
lead to different subspaces of $\M(d, \mathbb{C})$, these subspaces are unique up 
to conjugacy (and congruence) by invertible unitary matrices. }
\end{definition}
We discuss its differences with~\cref{def: induced bipartite graph} and why this 
is also a natural definition in~\cref{subsec:symmetries}. The induced 
correspondence between nil and acyclicity is the following, proved in~\cref{sec: 
vertex feedback}.
\begin{theorem}\label{thm:acyclic_restriction}
    Let $G=([n], E)$ be a directed graph and $\cS_G\leq\M(n, \mathbb{C})$ be the graphical 
matrix space associated with $G$ over $\mathbb{C}$. The maximum number of vertices in an acyclic induced subgraph of $G$ equals the largest dimension of $U \leq \mathbb{C}^n$ such that $\cS_G[U]$ is nil.
\end{theorem}

\subsubsection{Strong connectivity and irreducibility}\label{subsubsec:strong-conn-irred}

A directed graph $G=([n], E)$ 
is called \emph{strongly connected} if for any non-empty and proper $S\subset[n]$, there are 
arcs going out of $S$. Equivalently, $G$ is strongly connected if there is a directed path connecting any ordered pair of vertices.
A matrix space $\cS\leq\M(n, \F)$ is \emph{irreducible}, if it has no non-trivial invariant subspace; otherwise we call it reducible. Namely, for any non-zero and proper subspace 
$U<\F^n$, $\cS(U):=\langle \cup_{B\in \cS}B(U)\rangle$ is not contained in $U$. Similar to the properties of being acyclic and nil, we prove basic, inherited, and induced correspondences between strong connectivity and irreducibility. We begin with the basic correspondence.

\begin{theorem}[Simplified version of~\cref{thm:conn_s}]\label{thm:conn}
Let $G=([n], E)$ be a directed graph and $\cS_G\leq\M(n, \F)$ be the graphical 
matrix space associated with $G$. Then $G$ is strongly connected if and only if $\cS_G$ is 
irreducible. 
\end{theorem}

Next, we state the inherited correspondence between strong connectivity and 
irreducibility, whose proof is given in~\cref{subsec: arc strong connectivity}.
\begin{theorem}\label{thm:lambda}
Let $G=([n], E)$ be a directed graph and $\cS_G\leq\M(n, \F)$ be the graphical 
matrix space associated with $G$. The maximum number of arcs in a non-strongly connected subgraph of $G$ equals the largest dimension of a reducible subspace of $\cS_G$.
\end{theorem}

Finally, we state our induced correspondence result between strong connectivity 
and irreducibility, whose proof is in~\cref{sec: vertex strong connectivity}.
\begin{theorem}\label{thm:kappa}
    Let $G=([n], E)$ be a directed graph and $\cS_G\leq\M(n, \F)$ be the graphical 
    matrix space associated with $G$. The maximum number of vertices in a non-strongly connected induced subgraph of $G$ equals the largest dimension of $U \leq \mathbb{C}^n$ such that $\cS_G[U]$ is reducible.
\end{theorem}

Note that the minimum number of arcs that can be deleted from $G$ to make it not 
strongly connected is a well-known graph-theoretic parameter called the 
\emph{arc-strong connectivity} of $G$. Similarly, the minimum number of vertices 
whose deletion makes $G$ not strongly connected is well-known as the 
\emph{vertex-strong connectivity} of $G$. See e.g.~\cite[Chapter 1.5]{BG08}.
In~\cite{LQ19}, connections
between the vertex and edge connectivities of undirected graphs and certain 
parameters of 
alternating matrix spaces were established. The proofs of 
\cref{thm:lambda,thm:kappa} are non-trivial adaptations of those for 
\cite[Propositions 2.5 and 2.4]{LQ19}.

\subsubsection{Isomorphism, conjugacy, and congruence}\label{subsubsec:iso}

Let $G=([n], E)$ and $H=([n], F)$ be two 
directed graphs. We say that $G$ and $H$ are \emph{isomorphic}, if there exists a 
permutation $\sigma\in \S_n$, such that $(i, j)\in E$ if and only if $(\sigma(i), 
\sigma(j))\in F$. 

For matrix spaces, there are two natural notions which can play the role of
``isomorphism''.
Let $\cS_1, \cS_2\leq\M(n, \F)$ be two matrix spaces. We say that $\cS_1$ and $\cS_2$ are 
\emph{conjugate}, if 
there exists $T\in\GL(n, \F)$, such that $\cS_1=T\cS_2 T^{-1}:=\{TCT^{-1} \mid 
C\in\cS_2\}$. We say that $\cS_1$ and $\cS_2$ are \emph{congruent}, if there exists 
$T\in \GL(n, \F)$, such that $\cS_1=T\cS_2 T^t:=\{TCT^t \mid C\in 
\cS_2\}$.\footnote{When $\F=\mathbb{C}$, or more generally, $\F$ is a quadratic extension 
of a subfield, we may adopt the conjugate transpose $T^*$ instead of $T^t$.}
The following result establishes a basic correspondence between graph isomorphism 
and both notions of matrix space ``isomorphism''.

\begin{theorem}[Combined and simplified version of~\cref{prop:congruence,prop:conjugate}]\label{thm:iso}
Let $G=([n], E)$ and $H=([n], F)$ be two directed graphs and $\cS_G,\cS_H\leq\M(n, \F)$ be their graphical 
matrix spaces, respectively. The following are equivalent: 
\begin{enumerate}
\item $G$ and $H$ are isomorphic.
\item $\cS_G$ and $\cS_H$ are conjugate. 
\item $\cS_G$ and $\cS_H$ are congruent.
\end{enumerate}
\end{theorem}

The proof of the equivalence between $1$ and $3$ in \cref{thm:iso} is 
straightforward, and the proof strategy goes 
back to~\cite{HQ21}. We can actually prove a somewhat more general basic 
correspondence for congruence (see \cref{prop:congruence}), which states that $G$ is isomorphic to a subgraph 
of $H$ if and only 
if $\cS_G$ is congruent to a subspace of $\cS_H$. 

The proof of the equivalence between $1$ and $2$ in \cref{thm:iso} is elementary 
but much more complicated, and the proof strategy is inspired by 
\cite[Theorem 4.13 in arXiv version 2]{BS20}. 
Interestingly, the stronger correspondence mentioned in the previous paragraph fails 
in the conjugacy setting. That is, there exist graphs 
$G$ and $H$ such that $H$ is not isomorphic to any subgraph of $G$ but $\cS_H$ is 
conjugate to a subspace of $\cS_G$ (cf.~\cref{example: no embedding version of 
conjugacy}). This suggests a subtle difference between these two 
symmetry notions of matrix spaces.

Interestingly, the basic correspondence between isomorphism and congruence 
\emph{cannot} be boosted to an inherited correspondence, as seen in~\cref{example: no inherited version of congruence}.

\subsubsection{Vertex transitivity and conjugacy/congruence irreducibility} 
Let $G=([n], E)$ be 
a directed graph. We say that $G$ is \emph{vertex-transitive} if its automorphism 
group acts transitively on $[n]$. 
Recall that a matrix group $\cG\leq\GL(n, \F)$ is 
reducible\footnote{Note that this definition agrees with that of reducibility in~\cref{subsubsec:strong-conn-irred}: $\cG\leq \GL(n,\F)$ is reducible if and only if the linear subspace $\langle \cG\rangle \leq \M(n,\F)$ it spans is reducible as a matrix space.}, if there exists a non-zero and proper $U\leq \F^n$ such that for any $A\in \cG$, 
$A(U)\leq U$. Otherwise, we call $\cG$ \emph{irreducible}.
Let $\cS\leq\M(n, \F)$. 
Define $\Conj(\cS):=\{T\in \GL(n, \F) \mid T\cS 
T^{-1}=\cS\}\leq \GL(n, \F)$. 
We say that $\cS$ is \emph{conjugacy irreducible}, if $\Conj(\cS)$ is irreducible as a matrix group. Define $\operatorname{Cong}(\mathcal{S}):=\{T\in\operatorname{GL}(n, \mathbb{F}) \mid T\mathcal{S} T^{t}=\mathcal{S}\}\leq\operatorname{GL}(n, \mathbb{F})$. We say that $\mathcal{S}$ is \emph{congruence irreducible}, if $\operatorname{Cong}(\mathcal{S})$ is irreducible as a matrix group.  The following result establishes a basic correspondence between vertex transitivity and conjugacy/congruence irreducibility.
\begin{theorem}\label{thm:trans}
Let $\F$ be a field of order $>2$, and let $G=([n], E)$ be a directed graph. Then 
the following are equivalent:
\begin{enumerate}
	\item $G$ is vertex-transitive.
	\item $\cS_G$ is conjugacy irreducible.
	\item $\cS_G$ is congruence irreducible. 
\end{enumerate}
\end{theorem}
See~\cref{subsec:transitivity} for the proof. This generalizes 
the 
following result from 
quantum information. It is well-known that a directed graph can be embedded into a 
quantum channel. In~\cite{quasirandom}, Bannink, Bri\"{e}t, Labib, and Maassen 
showed that a directed graph is vertex-transitive if and only if the resulting 
quantum channel is ``irreducibly covariant''.~\cref{thm:trans} readily 
implies this result, and generalizes it from $\F=\mathbb{C}$ to $\F$ being any field of order $>2$.

\subsection{Perspective: symmetries of matrix spaces}\label{subsec:symmetries}
In the results above, it was sometimes natural to view the restricted support of a matrix space as a bipartite graph, and sometimes as a directed graph. Moreover, the corresponding linear-algebraic properties were defined in a number of different ways. These choices stem from different natural ways of viewing the matrix space $\M(n,\F)$, which in turn are associated with different natural group actions on them, which may preserve some properties but not others. We now explain these in detail.

Let $V$ and $W$ be two $n$-dimensional vector spaces over a field $\F$. Then there are (at least) three natural ways of viewing the matrix space $\M(n,\F)$. Associated to each of these interpretations is a natural group action, which may not preserve every natural property of matrices, but does preserve all the natural properties which are inherent to the interpretation. 
\begin{itemize}
    \item First, we may identify $\M(n,\F)$ with the set of linear maps $V \to W$. In this case, the natural group action on $\M(n,\F)$ is the left-right action of the group $\GL(n,\F) \times \GL(n,\F)$, which corresponds to changing the bases of $V$ and $W$. This group action preserves natural properties of linear maps $V \to W$, such as their rank.
    \item Secondly, we may identify $\M(n,\F)$ with the set $\End(V)$ of endomorphisms of $V$, i.e., linear maps from $V$ to itself. In this case, the natural group action is the conjugation action of $\GL(n,\F)$, which corresponds to changing the basis of $V$. This symmetry preserves additional properties, such as the spectrum, determinant, and nilpotency.
    \item Finally, we may identify $\M(n,\F)$ with the set $\Bil(V)$ of bilinear forms $V \times V \to \F$. In this case, the natural action is the congruence action of $\GL(n,\F)$, which arises from changing the basis of the Gram matrices associated with the bilinear forms. This symmetry preserves natural properties of bilinear forms, such as symmetry and non-degeneracy, but does not preserve other properties such as nilpotency.
\end{itemize}

Similarly, the natural group action on $n$-vertex directed graphs is that of $\S_n$ (permuting the vertices), and the natural group action on bipartite graphs with $m$ vertices on the left and $n$ vertices on the right in each part is that of $\S_m \times \S_n$. In each of our correspondences between graph theory and linear algebra, there is an underlying interpretation of $\M(n,\F)$, and correspondingly an underlying relationship between the symmetries. For example, in~\cref{thm:non-matchable_subgraph}, we view $\M(m \times n,\F)$ as the space of linear maps from an $m$-dimensional vector space to an $n$-dimensional vector space, with the symmetries being the left-right action of $\GL(m,\F) \times \GL(n,\F)$. Here, the (bipartite) graph property of matching size is preserved by the action of $\S_m \times \S_n$, and the corresponding matrix property of rank is preserved by the left-right action.

In contrast, in~\cref{thm:nil}, we view $\M(n,\F)$ as the space of endomorphisms of an $n$-dimensional vector space, with the conjugation action of $\GL(n,\F)$. This symmetry preserves the property of nilpotency, just as the corresponding action of $\S_n$ on the set of directed graphs preserves the property of being acyclic.

Finally, in the correspondence between directed graph isomorphism and matrix space congruence (\cref{thm:iso}), we use the third interpretation of $\M(n,\F)$, namely the identification with $\Bil(V)$. This is immediately apparent from the fact that the definition of congruence concerns the natural symmetries of this interpretation.

Note that the third interpretation naturally comes with a notion of 
restriction. 
Indeed, let $V$ and $W$ be vector spaces over $\F$, and $U_1\leq V$ and $U_2\leq 
W$ be two subspaces. Then a bilinear form $V\times W\to \F$ immediately restricts 
to a bilinear form $U_1\times U_2\to \F$.
This is why~\cref{def: induced bipartite graph} is natural for bilinear forms.
In contrast, for a linear map $f: V\to V$, one cannot restrict the codomain from $V$ 
to $U\leq V$ unless $U$ is an invariant subspace of $V$ under $f$.

This is relevant for our two induced correspondences, \cref{thm:acyclic_restriction,thm:kappa}. In both of them, we study properties of endomorphisms (nilpotency and irreducibility, respectively), and thus the symmetry must preserve these properties, which means that the underlying symmetry is the conjugation action of $\GL(n,\F)$. However, because they establish a correspondence with induced subgraphs, we also need a notion of restriction. But as discussed above, in general, restriction does not preserve these properties of endomorphisms, as the natural symmetry is that of $\Bil(V)$, rather than $\End(V)$. Thus, for these induced correspondences, we need a symmetry that is simultaneously a conjugation action and a congruence action. The natural such symmetry is that of the unitary group. This is why induced subspaces (\cref{def:induced}) are only defined when the underlying field is $\mathbb{C}$, and why~\cref{thm:acyclic_restriction,thm:kappa} are only stated over $\mathbb{C}$ as well: it is only when we have access to a symmetry that preserves \emph{both} the structure of $\End(V)$ and $\Bil(V)$ that we can expect such induced correspondences to hold. Moreover, this explains why~\cref{def:induced} is natural: by insisting that the rows of $T_U$ form an orthonormal basis of $U$, we are ensuring that the operation of passing to an induced subspace preserves all the relevant properties. 

As a final remark, we think it is interesting and noteworthy that in the first two interpretations of matrix spaces, we have examples of inherited correspondences (the rank-matching correspondence for the first interpretation, and the acyclicity-nilpotency and strong connectivity-irreducibility correspondences for the second interpretation). In contrast, under the third interpretation, we have an instance of a basic correspondence which \emph{cannot} be boosted to an inherited correspondence (isomorphism and congruence). It would be extremely interesting to classify which basic correspondences yield inherited correspondences and which ones do not; it seems possible that the first two interpretations are more amenable to such inherited correspondences.

\subsection{Perspective: computational complexity} \label{subsec:complexity}

\paragraph{Works based on the connection between matchings 
and ranks.} In 
\cref{subsec:overview}, we reviewed how the connection between perfect matchings 
and full-rank matrices served as a key to several important questions and results 
in computational complexity. Lov\'asz \cite{Lov79} used this correspondence to develop an RNC algorithm for the bipartite matching problem. The question of computing a perfect 
matching (if there exists one) in RNC was subsequently solved by Karp, Upfal and 
Wigderson~\cite{KUW86}, then simplified by Mulmuley, Vazirani and Vazirani~\cite{MVV87}. Recently, 
deterministic 
quasi-NC algorithms were devised for the perfect matching problem on bipartite 
graphs 
\cite{FGS19} and on general graphs~\cite{ST17}.

Based on the connection between bipartite matching and singularity, Edmonds \cite{Edmonds67} proposed the question of finding an efficient deterministic algorithm for determining if a matrix space contains a non-singular matrix (or, equivalently, if a symbolic matrix has a non-zero determinant). Over finite fields this problem is 
NP-complete~\cite{BFS99}, and over large enough fields this problem is in RP 
\cite{Lov79} by the Schwartz--Zippel lemma~\cite{Sch80,Zip79}. The converse 
problem, namely to decide if a matrix space contains only singular matrices, is 
now more commonly known as the symbolic determinant identity testing (SDIT) 
problem. To derandomize SDIT is a central problem in computational complexity. It 
is equivalent to the polynomial identity testing problem for algebraic branching 
programs~\cite{SY10}. Its central importance in computational complexity stems from the following surprising result of Kabanets and Impagliazzo \cite{KI04}: An efficient deterministic algorithm for SDIT implies strong circuit lower 
bounds which seem beyond the reach of current techniques. Underlying this is a foundational result of Valiant \cite{Val79}, who found a way of converting a short algebraic formula for a polynomial $f$ into a small symbolic matrix $B_f$ with $\det(B_f)=f$. An important aspect of Valiant's construction is that the symbolic matrix $B_f$ has a very restricted support; in fact, in his original paper, he explicitly viewed the entries of $B_f$ as living on the edges of a graph! This yields another reason to study matrix spaces with restricted support: They may be easier to understand, and nonetheless a deterministic algorithm for SDIT with certain restrictions on the support would yield a deterministic algorithm for polynomial identity testing. Additionally, it is possible that finding an algorithm for SDIT is easier in certain classes of restricted support, and studying such intermediate questions is a natural approach towards the general problem. One example of restricted support that is natural to study is that which captures \emph{graph rigidity}, itself a celebrated open problem (see e.g.\ \cite{RazWigderson}).

The non-commutative version of the SDIT problem, known as the non-commutative rank 
problem of symbolic matrices, has received considerable attention recently. It can 
be 
viewed as a linear algebraic analog of the problem of deciding if a bipartite 
graph has ``Hall's obstructions''~\cite{Hal35} for perfect matchings. Originally 
proposed by P. M. Cohn~\cite{Cohn75} in the context of free skew fields, this 
problem was recently shown to be in P via three solutions: over $\Q$ by 
\cite{GGOW16}, and over any field first by~\cite{IQS17,IQS18} and then by 
\cite{HH21}.

\paragraph{Discussions on the connection between cyclicity and nilpotency.}
After identifying the correspondence between cycles 
and nilpotent matrices, it is natural to examine the following questions. 
\begin{enumerate}
\item Compute a cycle in a directed graph in NC. 
\item Decide if a matrix space is nil, i.e., consisting of only nilpotent matrices. 
We call this the Symbolic Nil Testing (\MSNT) problem
\end{enumerate}

It turns out that the first problem reduces to computing a perfect matching on a 
bipartite graph in NC (\cref{prop:cyclicity_matching}), and the second problem 
reduces to the symbolic determinant 
identity testing (SDIT) problem (\cref{prop:nil_sing}). In fact, it can be shown 
that computing a cycle in a directed graph can be done in NC 
(\cref{prop:cycle_nc}). 

It is interesting to compare \MSNT{} and SDIT from the perspective of 
deterministic algorithms. 
For example, singular matrix spaces are preserved under the left-right action, and 
the nullcone of the left-right action on matrix tuples consists of those with a 
``shrunk subspace'' (a linear-algebraic analog of the ``Hall's obstruction'' for perfect matching). To determine if a matrix tuple lies in the nullcone of the 
left-right action is just the non-commutative rank problem, which is a rather 
non-trivial problem and we briefly 
touched upon this in the above.
To solve this problem in a black-box 
way is still open.

On the other hand, nil matrix spaces are 
preserved under the conjugation action, and the nullcone of the conjugation action 
on matrix tuples consists of those which can be simultaneously 
strict upper-triangularized. To determine if a matrix tuple $\vB=(B_1, \dots, 
B_m)\in\M(n, \F)^m$ can be simultaneously 
strict upper-triangularized can be solved by testing if
$\vB^n=\langle B_{i_1}B_{i_2}\cdots B_{i_n} \mid i_j\in[m]\rangle$ is  the zero 
matrix space. A black-box quasipolynomial-time algorithm for the Noether 
normalization of this action was shown by Forbes and Shpilka~\cite{FS13}. 

As another example, consider those matrix spaces that have a basis of rank-1 
matrices. In this case, SDIT can be solved in deterministic polynomial time, again 
with non-trivial algorithms~\cite{Gur04,IKS10,IKQS15}. For \MSNT, this is easy 
because it is also strictly upper-triangularizable \cite{MATHES1991215}, so the 
algorithm mentioned in the above paragraph suffices. 

To summarize, \MSNT{} seems easier than SDIT for most special cases. Given these, 
it may be surprising to note that SDIT reduces to a quantitative version of 
\MSNT{} as follows. Recall that a symbolic matrix is a matrix whose entries are 
affine linear forms. Let $B$ be a symbolic matrix. SDIT asks if $\det(B)$ is the 
zero polynomial. The nilpotency index of $B$ is the smallest $k$ such that $B^k$ 
is the zero matrix. 
\begin{proposition}\label{prop:SDIT_NilIndex}
SDIT reduces to the problem of deciding if the nilpotency index of a symbolic 
matrix is no more than a given integer.
\end{proposition}
\begin{proof}
Let $B$ be a symbolic matrix of size $n\times n$. By \cite{Ber84}, there exist symbolic matrices $C_1, 
\dots, C_\ell$ of 
size $t$, $\ell, t=\poly(n)$, such that the $(1, 1)$ entry of $C_1\cdot \dots 
\cdot C_\ell$ is $\det(B)$. Let $C_0=C_{\ell+1}=\E_{1,1}$. Construct a 
$(\ell+3)\times (\ell+3)$ block matrix $T$ 
with the block sizes being $t$. For $i\in[\ell+2]$, the $(i, i+1)$th block of $T$ 
is $C_{i-1}$. We then see that the $(1, \ell+3)$th block of $T^{\ell+2}$ is 
$C_0\cdot C_1\cdot \dots \cdot C_{\ell+1}$, which is zero if and only if $\det(B)$ 
is the zero polynomial. So the nilpotency index of $T$ is $\leq \ell+2$ if and 
only if $\det(B)$ is the zero polynomial. This concludes the proof.
\end{proof}
Note that the symbolic matrix constructed in the proof of 
\cref{prop:SDIT_NilIndex} is clearly nilpotent. Therefore, it does not give us a 
reduction from SDIT to \MSNT{}.

From the above discussion, at 
present it looks to us that solving the general \MSNT{} is a good challenge. This 
is because classifying nil matrix spaces is a wild problem~\cite{Ger58} with 
implications~\cite{VF17} to the classical Albert's problem~\cite{Alb50}. This is 
also motivated by the work of Makam and Wigderson~\cite{MW21}, 
who showed that the matrix tuples spanning singular matrix spaces do not form a 
nullcone in general. Therefore, it is interesting to identify further algebraic 
varieties for which the membership problem is interesting enough as the test bed 
for further progress towards solving SDIT. 

\paragraph{Applications: NP-hardness for some linear-algebraic problems.}
One consequence of our inherited and induced correspondences is that we can prove NP-hardness results for several linear-algebraic problems, using known NP-hardness results for certain graph-theoretic problems. Namely, we have the following results.
\begin{theorem}
Let $\cS\leq\M(n, \F)$ be of dimension $m$ whose matrices have rank at most $r$. 
Let $d\in[m]$. To decide if $\cS$ admits a dimension-$d$ subspace whose matrices 
have rank at most $r-1$ is NP-hard over any field, and NP-complete over finite 
fields.
\end{theorem}
\begin{proof}
For a bipartite graph $G=(L\cup R, E)$, a $k$-blocker is a subgraph of $G$ whose 
matching number is at most the matching number of $G$ minus $k$. It was shown 
in~\cite[Theorem 3.3]{ZENKLUSEN20094306} that deciding the existence of a size-$d$ 
$1$-blocker in bipartite graphs is NP-complete. The theorem then follows 
with~\cref{thm:non-matchable_subgraph}.
\end{proof}
\begin{theorem}
Let $\cS\leq\M(n, \F)$ be of dimension $m$, and $d\in[m]$. 
To decide if $\cS$ admits a dimension-$d$ nil subspace is NP-hard 
over any field, and NP-complete over finite fields. 
\end{theorem}
\begin{proof}
This follows from~\cref{thm:acyclic_subgraph} and Karp's classical result that the feedback arc set problem is NP-complete 
\cite{Kar72}.
\end{proof}

\begin{theorem}
    Let $\cS\leq\M(n, \mathbb{C})$, and $d\in[n]$. 
    To decide if $\cS$ admits a $d$-dimensional $U\leq \mathbb{C}^n$ such that $\cS[U]$ is 
    nil is NP-hard. 
\end{theorem}
\begin{proof}
This follows from~\cref{thm:acyclic_subgraph}
and Karp's well-known result that the 
feedback vertex set problem is NP-complete 
\cite{Kar72}.
\end{proof}

\begin{theorem}
Let $\cS_1, \cS_2\leq\M(n, \F)$. 
To decide if $\cS_1$ contains a congruent copy of $\cS_2$ as a subspace is 
NP-hard over any field, and NP-complete over finite fields. 
\end{theorem} 
\begin{proof}
This follows from the NP-completeness of the subgraph isomorphism problem (via 
e.g. Hamiltonian cycles of paths \cite{Kar72}) and \cref{prop:congruence} (mentioned after~\cref{thm:iso}).
\end{proof}

\subsection{Connections with quantum information theory} \label{subsec:quantum}
For this subsection, we restrict to the case $\F=\mathbb{C}$.
Graph theory 
has inspired several research topics in quantum information. For example, Hastings~\cite{PhysRevA.76.032315}, and Ben-Aroya, Schwartz, and Ta-Shma~\cite{BST10} 
introduced quantum expanders which are quantum channels satisfying certain regular 
and expanding properties. Roughly speaking, they view quantum channels as generalizations of the adjacency matrices of graphs. This leads to other research lines linking quantum information theory with graph theory, which indicate that some spectral
properties of quantum channels properly generalize certain graph-theoretic properties (of adjacency matrices). 

As another example, in the study of zero-error communications of quantum channels, Duan, Severini and Winter in  
\cite{duan2013} associated each quantum channel an operator system, which is analog to the classical setting where Shannon associated each classical channel with a graph~\cite{MR0089131}. This gives rise to another way to generalize graphs to matrix spaces. Along with this generalization, Duan, Severini, and Winter proposed quantum generalizations of independence number and the Lov\'asz theta function for operator systems. Viewing operator systems as generalizations of graphs has led to a fruitful research line which connects combinatorics, operator algebra and optimization theory; see~\cite{Wea21} and the references therein.

We point out that the results presented in this paper recover and generalize several 
interesting results relating to the above two types of generalizations. Thus, we 
believe that our framework provides a more systematical way to study connections 
between graphs and matrix space. To be more specific, we show 
that~\cref{thm:conn,thm:iso,thm:trans} can be applied to the quantum settings, 
which reveals that 
\begin{itemize}
\item the irreducibility of quantum channels (introduced in~\cite{EH78,QCO}) is a 
generalization of the strong connectivity of directed graphs~(\cref{cor: 
equivalence quantum classical irreducibility});
\item the irreducibly covariant quantum channels are generalizations of vertex 
transitive graphs~(\cref{cor: vertex-transitive}, first proved 
in~\cite{quasirandom});
\item the connectivity of operator systems (introduced 
in~\cite{CHAVEZDOMINGUEZ202137}) is a generalization of connectivity of 
(undirected) graphs~(\cref{cor: conn}, first proved 
in~\cite{CHAVEZDOMINGUEZ202137});
\item the isomorphism between operator systems (introduced 
in~\cite{ORTIZ2015128}) is a generalization of the isomorphism between graphs 
(\cref{cor: isomorphism}, first proved in~\cite{ORTIZ2015128}).
\end{itemize}
We emphasize that all the above generalizations are already shown in the 
references. However, our results simplify some of the proofs and obtain 
more general results. See~\cref{sec: quantum} for a more detailed discussion.

\subsection{Some known connections between graphs and matrix 
spaces}\label{subsec:further-connections}
Inspired by the classical 
connection between perfect matchings and full-rank matrices, 
several correspondences between 
graph-theoretic structures, and structures for alternating matrix spaces, have 
been 
 discovered recently, including:
\begin{itemize}
\item Independent sets vs. (totally) isotropic spaces, and vertex colorings vs 
(totally) isotropic 
decompositions~\cite{BCG+19}.
\item 
Connectivity vs. orthogonal indecomposability. As a consequence, 
correspondences of vertex and edge connectivities for alternating matrix spaces 
are also presented~\cite{LQ19}.
\item Isomorphism notions for graphs and alternating matrix spaces~\cite{HQ21}.
\end{itemize}

Some graph-theoretic questions were also translated to the matrix space setting, 
including:
\begin{itemize}
\item Transferring techniques for graph isomorphism to study matrix space congruence 
\cite{LQ17}.
\item Enumeration formulas of isotropic spaces and orthogonal decompositions as 
$q$-analogs of enumeration formulas of independent sets and connected graphs 
\cite{Qia21}.
\item Tur\'an and Ramsey problems for alternating matrix spaces 
\cite{Qia20_extremal}.
\item New graph polynomials from group zeta functions, through the connection between alternating matrix spaces and nilpotent $p$-groups~\cite{Rossmann19,Rossmann21}.
\end{itemize}

Some results and techniques in the above works inspire results in the present 
work. One key difference is that one emphasis in this paper is to examine the 
connection between \emph{directed graphs} and matrix spaces, whereas most works in 
the above-studied connections between undirected graphs and bipartite graphs and 
matrix spaces.

\subsection{Discussion, research directions, and open problems}\label{subsec:open}
As we have attempted to demonstrate in the introduction, there is a surprising wealth of correspondences between graph theory and matrix spaces. Our goal in this paper is to initiate a systematic study of such correspondences, and we hope that this will lead to fruitful future work. One big ``meta-question'' that we are interested in is a further development of this theory. Which other natural graph-theoretic properties have linear-algebraic analogs? Conversely, which linear-algebraic properties correspond to graph-theoretic ones in matrix spaces of restricted support? More broadly, can one develop a general theory and characterize the family of such properties? Additionally, is it possible to characterize which basic correspondences can be boosted to inherited and/or induced correspondences? Given that two of our inherited correspondences (\cref{thm:non-matchable_subgraph,thm:nil}) generalize classical results on matrix spaces (Dieudonn\'e's theorem \cite{Die48} and Gerstenhaber's theorem \cite{Ger58}, respectively), it is possible that such a general theory could imply or generalize important results in algebra.

It would also be interesting to find further applications of such correspondences. 
As mentioned above, the basic correspondence between bipartite perfect matchings 
and singularity of matrix spaces has been immensely fruitful, yielding in 
particular an efficient randomized parallel algorithm for perfect matching 
\cite{Lov79}. 
As discussed in \cref{subsec:complexity}, our new results allow us 
to prove NP-hardness 
results for certain linear-algebraic questions, and lead to the matrix space nil 
test problem which is rather interesting in the context of polynomial identity 
testing.
In \cref{subsec:quantum}, we explained how our results generalize certain results 
in quantum information theory.
The reason for applications in quantum information is that matrix tuples and 
matrix spaces are used to set up superoperators there. 
Since matrix spaces arise in numerous mathematical areas as indicated at the 
beginning of \cref{subsec:overview}, with the number and depth of the 
correspondences that are now known to exist, it seems likely that there are 
further applications waiting to be discovered.

Finally, we mention one specific conjecture whose proof or disproof we would be extremely interested in. Recall that~\cref{thm:nil} gives a basic correspondence between the properties of being acyclic and nil. There is a simple natural generalization of this, which we prove in~\cref{subsec: quantatitative}, which says the following for a directed graph $G=([n],E)$ and any non-negative integer $k$. Every collection of disjoint cycles in $G$ covers at most $k$ vertices if and only if every matrix in $\cS_G$ has at most $k$ non-zero eigenvalues. Note that the case $k=0$ corresponds precisely to $G$ being acyclic, and to $\cS_G$ being nil. Our conjecture is that this basic correspondence can be boosted to an inherited correspondence.
\begin{conjecture}\label{conj:atkinson}
    Let $G=([n],E)$ be a directed graph, and let $0 \leq k \leq n$ be an integer. The maximum number of edges in a subgraph $H$ of $G$ in which every collection of disjoint cycles covers at most $k$ vertices equals the largest dimension of a subspace of $\cS_G$ in which every matrix has at most $k$ non-zero eigenvalues.
\end{conjecture}
If true, \cref{conj:atkinson} would generalize a well-known theorem of Atkinson 
\cite{Atk80}, who proved that if $\F$ is sufficiently large and $S\leq \M(n,\F)$ 
is a matrix space in which every matrix has at most $k$ non-zero eigenvalues, then 
$\dim S \leq nk + \binom {n-k}2$. This is precisely the statement 
of~\cref{conj:atkinson} in case $G$ is a complete directed graph, i.e., 
\cref{conj:atkinson} is the restricted-support version of Atkinson's full-support 
theorem. Note too that~\cref{conj:atkinson} generalizes~\cref{thm:nil}, just as 
Atkinson's theorem generalizes Gerstenhaber's theorem.

Some algorithmic problems about matrix spaces are also worth studying. For example, 
\cref{thm:lambda} naturally leads to the problem of finding the maximum dimension 
over reducible subspaces of a matrix space. The corresponding graph-theoretic 
problem, namely the min-cut problem for directed graphs, can be solved in 
deterministic polynomial time \cite{BG08}. The submodular optimization algorithms 
over modular lattices \cite{HH21} may be relevant for this purpose. 

\section{Preliminaries}\label{sec: preliminary}

We collect basic notation. More definitions and notation will be introduced in 
each section. 

For $n\in\N$, $[n]:=\{1, 2, \dots, n\}$. 

\paragraph{Graphs.} A directed graph is $G=(V, E)$, where $V$ is the vertex set 
and $E\subseteq V\times V$ is the \emph{arc} set. We shall mostly working with 
directed 
graphs whose vertex set is $V=[n]$. 

A bipartite graph is $G=(L\cup R, E)$, where $L$ and $R$ are the left and right 
vertex sets, and $E\subseteq L\times R$ is the \emph{edge} set. We shall mostly 
working 
with bipartite graphs with $L=R=[n]$. 

An undirected graph is $G=(V, E)$, where $V$ is the vertex set and $E\subseteq 
\binom{V}{2}:=\{\{v, v'\} \mid v, v'\in V, v\neq v'\}$ is the \emph{edge} set. 

We say two graphs $G=([n], E)$ and $H=([n], F)$ are isomorphic if there is a bijection $f:[n]\to[n]$ such that $(i,j)\in E~\Leftrightarrow~(f(i),f(j))\in F$.

We denote $E^{\prime}$ being a subset of $E$ as $E^{\prime} \subseteq E$ and $E^{\prime}$ being a \emph{proper} subset of $E$ as $E^{\prime} \subset E$. A spanning subgraph $G'=(V,E')$ of $G=(V, E)$ has the same vertex set $V$ and edge 
set $E'\subseteq E$. Sometimes we call spanning subgraph as subgraph if there is 
no 
confusion. An induced subgraph $G[V']=(V',E')$ of $G=(V, E)$ has the vertex set 
$V'\subseteq V$, and $E'=\{(v,v')\in E\mid v,v'\in V'\}$. 

\paragraph{Vector spaces.} Let $\F$ be a field. Let $\F^n$ be the vector space of 
length-$n$ \emph{row} vectors. The $i$th standard basis vector of $\F^n$ is 
denoted $e_i$. 
We use $U\leq\F^n$ to denote that $U$ is a subspace of $\F^n$ (and use $<$ for \emph{proper} subspace). For a 
set of vectors $\{v_1,\dots,v_m\}$, let $\langle v_1,\dots,v_m\rangle\leq\F^n$ be their linear 
span.

\paragraph{Matrices.} Let $\M(n, \F)$ be the linear space of $n\times n$ matrices over $\F$. Given $(i, j)\in[n]\times [n]$, the elementary matrix $\E_{i,j}\in \M(n, \F)$ is the matrix with the $(i, j)$th entry being $1$, and other entries $0$. We use $I_n$ to denote the $n\times n$ identity matrix. 

For a matrix $B\in \M(n,\F)$, let $\rowspan(B)$ be the vector space spanned by 
the rows of $B$ and $\colspan(B)$ be the vector space spanned by the 
columns of $B$. We have $\rank(B)=\dim(\colspan(B))=\dim(\rowspan(B))$. A matrix 
$B\in \M(n,\F)$ is nilpotent, if $B^k=0$ for some $k\in\N$. For $i, 
j\in[n]$, $B(i, j)$ denotes the $(i, j)$th entry of $B$. 

For a matrix $B\in \M(n,\F)$ and a vector $v\in\F^n$, we usually consider the 
right action of $B$ on $v$, and denote the result as $vB$. The 
reason to use row vectors and matrices acting on the right is to be consistent 
with directed graphs: if we use $\E_{i,j}$ to represent the arc from $i$ to $j$, 
then the right action of $\E_{i,j}$ sends $e_i$ to $e_j$.
For a subspace $U\leq \F^n$, let $B(U):=\{uB\mid u\in U\}\leq \F^n$. 

\paragraph{Matrix tuples and matrix spaces.} A matrix tuple $\tS$ of length $m$ is 
an element of $\M(n, \F)^m$. A \emph{matrix space} $\cS$ is a linear subspace of 
$\M(n, \F)$. Given $\cS\leq\M(n, \F)$ and $R, T\in\M(n, \F)$, $R\cS T:=\{RST\mid 
S\in \cS\}\leq \M(n, \F)$. 
Two matrix spaces $\cS,\hat{\cS}\leq \M(n,\F)$ are \emph{conjugate} if 
there exists $T\in\GL(n,\F)$ such that $T\cS=\hat{\cS}T$. They are 
\emph{congruent} if there exists $T\in\GL(n,\F)$ such that 
$T\cS T^t=\hat{\cS}$.

\section{Matchings in bipartite graphs and ranks of matrices}
\label{sec:matching}

\subsection{The inherited correspondence}\label{subsec:ranks-inherited}

In this subsection, we prove~\cref{thm:non-matchable_subgraph}, establishing an inherited correspondence between the matching size in a bipartite graph and the rank of matrices in its graphical matrix space.

Let $G=([m]\cup[n],E)$ be a bipartite graph with $m\leq n$. 

A matching in $G$ is defined as a subset $M\subseteq 
E$ such that for any $(i_1,j_1),(i_2,j_2)\in M$, $i_1\neq i_2$ and $j_1\neq j_2$. 

Let $B\in\M(m\times n,\F)$ be the bipartite adjacency matrix of $G=([m]\cup[n],E)$, where $B(i,j)=1$ if and only if $(i,j)\in E$; and $B(i,j)=0$ otherwise. Let $\rho(B)$ be the minimum number of columns and rows which cover all nonzero entries of $B$. Then K\"onig's theorem, a fundamental result in graph theory, says that the matching number of $G$ equals $\rho(B)$ (e.g. see~\cite[Theorem 2.1.1]{Die17}).

For $(i_1,j_1),(i_2,j_2)\in [m] \times[n]$, the lexicographic order $\prec$ is 
defined as $(i_{1}, j_{1})\prec(i_{2}, j_{2})$ if and only if $i_{1}< i_{2}$ or 
$i_{1}=i_{2}$ and $j_{1}<j_{2}$.
For $A\in \M(m\times n,\F)$, let $p(A) =\min \{(i, j): A(i,j) \neq 0\}$ be 
the position of $A$'s lexicographically first non-zero entry. An important result 
in~\cite{Mes85} is the following:

\begin{lemma}[\hspace{1sp}{\cite[Theorem 1]{Mes85}}]\label{Meshulam}
    Given matrices $A_1,\dots,A_d\in \M(m\times n,\F)$, construct $B\in \M(m\times 
    n,\F)$ to have $1$ in all the positions $p(A_1),\dots,p(A_d)$ and zeros 
    elsewhere. Then $\langle A_1,\dots,A_d\rangle$ contains a matrix of rank at least
    $\rho(B)$.
\end{lemma}

Let $\mnss(G,r)$ be the maximum size of a spanning subgraph of $G$ whose matching number is at most $r$ and $\lsss(\cS_{G},r)$ be the largest dimension of a subspace of $\cS_{G}$ in which any matrix has rank at most $r$. We are ready to establish the following result.

\paragraph{\cref{thm:non-matchable_subgraph}, restated.}
    \textit{Let $G=([m]\cup [n], E)$ be a bipartite graph with $m\leq n$, and $\cS_G\leq\M(m\times n, 
    \F)$ be the graphical matrix space associated with $G$. Then for each $r\in[m]$, $\mnss(G, r)=\lsss(\cS_G, r)$.}

\begin{proof}
    To see $\mnss(G,r)\leq\lsss(\cS_{G},r)$, take a 
    maximum spanning subgraph $G'$ with matching number $\leq r$. By~\cref{thm:edm67}, we know that any matrix in $\cS_{G^{\prime}}$ has rank $\leq r$. Thus, $\mnss(G,r)=\dim(\cS_{G^{\prime}})\leq\lsss(\cS_{G},r)$.
    
     To see $\lsss(\cS_{G},r)\leq\mnss(G,r)$, take the 
     largest subspace $W$ of $\cS_{G}$ in which any matrix has rank $\leq r$. Let $d=\dim(W)$. Without loss of generality, we choose a basis $A_1,\dots,A_d$ of $W$ such that 
     $p(A_1),\dots,p(A_d)$ are all distinct. This can be done by Gaussian elimination (and viewing $A_1,\dots,A_d$ as $mn$-dimensional vectors). Construct $B$ as in~\cref{Meshulam} and let $H$ be the bipartite graph whose bipartite adjacency matrix is $B$. By~\cref{Meshulam}, $W=\langle A_1,\cdots,A_d\rangle$ contains a matrix of rank at least $\rho(B)$, which is the matching number of $H$. Thus, the matching number of $H$ cannot be greater than $r$, otherwise $W$ would contain a matrix of rank greater than $r$. Since $B$ has exactly $d$ nonzero entries, $H$ has exactly $d$ edges. Moreover, since each $A_i$ was supported on the edges of $G$, $B$ is supported on the edges of $G$ as well, and thus $H$ is a spanning subgraph of $G$. These imply that $\lsss(\cS_{G})\leq\mnss(G)$.
\end{proof}
\begin{remark}\label{remark: dfm thm}
\cref{thm:non-matchable_subgraph} can be seen as a generalization of the Dieudonn\'{e}--Flanders--Meshulam theorem~\cite{Die48,Fla62,Mes85}, which states that the largest dimension of a matrix space of $\M(m\times n,\F)$ ($m\leq n$) with bounded rank $r$ is $rn$. Taking $G$ to be the complete bipartite graph, then $\cS_G=\M(m\times n,\F)$. In this case, $\lsss(\cS_G,r)$ is given by the maximum size of spanning subgraphs of $G$ whose matching number is $r$, which is exactly $rn$. This recovers the Dieudonn\'{e}--Flanders--Meshulam theorem.
\end{remark}

\subsection{The induced correspondence}\label{subsec:ranks-induced}
In this subsection, we prove~\cref{thm:non-matchable_induced_subgraph}, establishing an induced correspondence between the matching size in a bipartite graph and the rank of matrices in its graphical matrix space.

Suppose we have $\cS\leq\M(m\times n, \F)$, $L\leq \F^m$, $R\leq\F^n$, $\dim(L)=s$ 
and $\dim(R)=t$. The \emph{order} of $\cS$ is $m+n$. Let $T_L$ (resp.\ $T_R$) be an $s\times m$ (resp.\ $t\times n$) 
matrix whose rows span $L$ (resp.\ $R$). Define $\cS[L, R]:=\{T_LBT_R^t\mid B\in 
\cS\}$ be the induced subspace of $\cS$ with respect to $L$ and $R$. 

Let $r\in \N$. For a graph $G$, let $\MaxBdMatOrder(G, r)$ be the maximum order 
over induced subgraphs of $G$ with matching number at most $r$. For a matrix 
space $\cS$, let $\MaxBdRankOrder(\cS, r)$ be the maximum order over induced 
subspaces of $\cS$ with maximum rank at most $r$. 

\paragraph{\cref{thm:non-matchable_induced_subgraph}, restated.}
\textit{Let $G=([m]\cup [n], E)$ be a bipartite graph with $m\leq n$, and $\cS_G\leq\M(m\times n, \F)$ be the graphical matrix space associated with $G$. Then for each $r\in[m]$, $\MaxBdMatOrder(G, r)=\MaxBdRankOrder(\cS_G, r)$.}
\begin{proof}
	To see $\MaxBdMatOrder(G, r)\leq\MaxBdRankOrder(\cS_G, r)$, let $V_1\cup V_2$ where $V_{1}\subseteq[m]$ and $V_{2}\subseteq[n]$ be vertex subsets of maximum order such that the matching number of $G[V_1\cup V_2]$ is at most $r$. Take $L=\langle e_i\mid i\in V_{1}\rangle\leq \F^m$ and $R=\langle e_i\mid i\in V_{2}\rangle\leq \F^n$. Also take an $s\times m$ matrix $T_{L}$ whose rows are exactly $\{e_i\mid i\in V_{1}\}$ and a $t\times n$ matrix $T_{R}$ whose rows are exactly $\{e_i\mid i\in V_{2}\}$. By~\cref{thm:edm67}, any matrix in $\cS_G[L,R]$ is of rank at most $r$, as $\cS_G[L,R]=\cS_{G[V_{1}\cup V_{2}]}$ and any matching of $G[V_{1}\cup V_{2}]$ has size at most $r$.

	To see $\MaxBdRankOrder(\cS_G, r)\leq\MaxBdMatOrder(G, r)$, let $L\leq\F^m$ with $\dim(L)=s$, $R\leq\F^n$ with $\dim(R)=t$, and $\dim(L)+\dim(R)=s+t=\operatorname{\MaxBdRankOrder}(\cS_G, r)$ such that $\cS_G[L,R]$ is an induced subspace of $\cS_G$ in which any matrix has rank at most $r$. Let $T_L$ be an $s\times m$ matrix whose rows span $L$ and $T_R$ be a $t\times n$ matrix whose rows span $R$. Then there exists $V_1\subseteq [m]$ (resp.\ $V_2\subseteq[n]$) such that the columns in $T_L$ (resp.\ $T_R$) with indices in $V_1$ (resp.\ $V_2$) are linearly independent. 
    We want to show that the matching number of $G[V_1\cup V_2]$ is at most $r$. Let $A$ (resp.\ $B$) be the $s\times s$ (resp.\ $t\times t$) submatrix of $T_L$ (resp.\ $T_R$) with column indices in $V_1$ (resp.\ $V_2$). This ensures that $A$ and $B$ are both full-rank.

    Now we claim that $A\cS_{G[V_1\cup V_2]}B^t$ is a subspace of $\cS_G[L, R]$. To see this, we denote by $E_L$ (resp.\ $E_R$) the $m\times s$ (resp.\ $n \times t$)  matrix formed of columns $\{e_i^t:i\in V_1\}$ (resp.\ $\{e_i^t:i\in V_2\}$) such that $A = T_LE_L$ (resp.\ $B^t = E_R^t T_R^t$). Note that $\cS_{G[V_1\cup V_2]}=E_L^t\cS_GE_R$ and $A\cS_{G[V_1\cup V_2]}B^t=T_LE_LE_L^t \cS_G E_R E_R^t T_R^t$. Since $E_LE_L^t$ and $E_RE_R^t$ are both diagonal, we have that $E_L E_L^t \cS_G E_R E_R^t \leq \cS_G$, implying $A\cS_{G[V_1\cup V_2]}B^t \leq T_L \cS_G T_R^t = \cS_G[L, R]$.
    It follows that the maximal rank of matrices in $\cS_{G[V_1\cup 
    V_2]}$ is at most the maximal rank of matrices in $\cS_G[L, R]$, which is at most $r$. By~\cref{thm:edm67}, $G[V_1\cup V_2]$ is an induced subgraph of order-$(s+t)$ whose matching number is at most $r$, which concludes the proof.
\end{proof}

\section{Cycles in directed graphs and nilpotent matrices}
\label{sec:nil}

In this section, we prove correspondences between cycles in directed graphs and nilpotency of matrices in their associated graphical matrix spaces. We begin with the basic correspondences, then prove the inherited correspondence in~\cref{subsec:acyclic-inherited} and the induced correspondence in~\cref{sec: vertex feedback}.

\subsection{The basic correspondences}

\subsubsection{Maximum walk lengths and nil/nilpotent indices}

Let $G=([n], E)$ be a directed graph. A \emph{walk of length $k$} in $G$ is a 
sequence of $k+1$ vertices $v_0,v_1,\dots,v_k\in[n]$, where $(v_{i-1},v_{i})\in E$ 
for all $i\in[k]$. We call $v_0$ and $v_k$ the starting and ending vertices, 
respectively, and $v_1,\dots,v_{k-1}$ are called the intermediate vertices. 
A \emph{path} in $G$ is a walk with no repeated vertices. 
A \emph{cycle} in $G$ is a path $v_0,\dots,v_k$ with $v_0=v_k$. We 
say that $G$ is \emph{cyclic} if there exists a cycle in $G$, and \emph{acyclic} 
otherwise.

Let 
$\mwl(G)$ be the maximum length of walks in a directed graph $G=([n], E)$. If $G$ is cyclic, define
$\mwl(G)=\infty$. From a walk of length $n$ in $G$ we can obtain a cycle 
in $G$. Thus for an acyclic graph $G$, $\mwl(G)\leq n-1$. Moreover, if $G$ is 
acyclic, every walk in $G$ is a path. We shall relate cycles in directed graphs with nilpotent matrices in matrix 
spaces. 
\begin{definition}
A matrix space $\cS\leq\M(n, \F)$ is \emph{nil}, if any $B\in \cS$ is a nilpotent matrix. The \emph{nil index} of $\cS$, $\nil(\cS)$, is defined as the smallest integer $k$ such that $B^k=0$ for any $B\in\cS$. If $\cS$ is not nil, then $\nil(\cS)=\infty$.

A matrix space $\cS\leq\M(n, \F)$ is \emph{nilpotent}, if there exists $k\in \N$, such that for any $B_1, 
\dots, B_k\in \cS$, $\prod_{i\in[k]}B_i=0$. The smallest such $k$ is called the \emph{nilpotent index} of $\cS$, denoted by $\nilp(\cS)$. If $\cS$ is not nilpotent, then $\nilp(\cS)=\infty$.
\end{definition}
Clearly, nilpotent matrix spaces are also nil. Thus $\nil(\cS)\leq\nilp(\cS)$ for 
any matrix space $\cS$. The converse is false; for example, it is easy to verify that the space $\left\{\smat{0&x&0\\y&0&x\\0&-y&0}:x,y \in \F\right\}$ is nil but not nilpotent.

The following theorem establishes the basic correspondence between the properties of 
being acyclic, nilpotent, and nil. In particular, it recovers~\cref{thm:nil} in~\cref{subsubsec:nil}.
\begin{theorem}\label{thm:acyclicity}
Let $G=([n], E)$ be a directed graph, and let $\cS_G\leq\M(n, \F)$ be the graphical matrix 
space associated with $G$. Then the following are equivalent: 
\begin{enumerate}
\item[$(1)$] $G$ is acyclic;
\item[$(2)$] $\cS_G$ is nilpotent;
\item[$(3)$] $\cS_G$ is nil.
\end{enumerate}
Furthermore, we have $\mwl(G)+1=\nil(\cS_G)=\nilp(\cS_G)$. 
\end{theorem}

\begin{proof}
\paragraph{(1) $\Rightarrow$ (2):} Suppose $G$ is acyclic. This implies that $G$ has a topological sort, that is, a linear ordering of the vertices such that all edges are directed from left to right. Thus, there exists a 
permutation matrix $P$, such that $P^{-1}\cS_G P$ is strictly 
upper-triangular. Then $P^{-1} \cS_GP$ is nilpotent, which implies that $\cS_G$ is 
also nilpotent. 

\paragraph{(2) $\Rightarrow$ (3):} This is straightforward by their definitions.

\paragraph{(3) $\Rightarrow$ (1):} We show that if $G$ is cyclic, then $\cS_G$ is 
not nil. Up to relabeling the vertices, suppose the vertex sequence 
$(1,2,3,\dots,k,1)$ forms a cycle in $G$, for some $k\geq 1$. Then the following 
matrix is in $\cS_G$: $B=\begin{bmatrix}B_0 & 0\\ 0 & 0\end{bmatrix}$ where 
$B_0\in 
\M(k,\F)$ is 
\[
B_0=\begin{bmatrix}
0& 1 &  &  & \\
& 0 & 1 &  & \\
&  & \ddots & \ddots & \\
&  &  & 0 & 1\\
1 &  &  &  & 0
\end{bmatrix}_{k\times k}.
\]
Clearly, $B_0$ is not nilpotent. It follows that $B$ is not nilpotent and $\cS_G$ is not nil.

The above shows that if $G$ has a cycle, then $\mwl(G)$, $\nil(\cS_G)$ and $\nilp(\cS_G)$ are all infinite. Thus $\mwl(G)+1=\nil(\cS_G)=\nilp(\cS_G)$ holds whenever $G$ is cyclic.

We now prove that $\mwl(G)+1=\nilp(\cS_G)=\nil(\cS_G)$ when $G$ is acyclic. In this case, $\mwl(G)$, $\nil(\cS_G)$, and $\nilp(\cS_G)$ are finite. 

\paragraph{For $\mwl(G)+1\geq \nilp(\cS_G)$.} Let $r=\mwl(G)$. Note that 
for a sequence of $r+1$ elementary matrices, their product is non-zero if and only 
if their indices form a walk of length $r+1$. Since there exist no walks of length 
greater than $r$ in $G$, any product of a 
sequence of $r+1$ such elementary matrices, which come from a linear basis of $\cS_G$, must be 
zero. It follows 
that $r+1\geq \nilp(\cS_G)$.

\paragraph{For $\nilp(\cS_G)\geq \nil(\cS_G)$.} This is straightforward by their 
definitions.

\paragraph{For $\nil(\cS_G)\geq \mwl(G)+1$.} Let $r=\mwl(G)$. We construct a 
matrix $B\in \cS_G$, such that $B^r\neq 0$, which would imply that 
$\nil(\cS_G)\geq r+1$. Since $r$ is finite, every walk in $G$ is a path. By 
relabeling the vertices, suppose $(1, 2, \dots, r+1)$ is a path of length $r$ in 
$G$. 
Then the following 
matrix is in $\cS_G$: $B=\begin{bmatrix}B_0 & 0\\ 0 & 0\end{bmatrix}$ where 
$B_0\in 
\M(r+1,\F)$ is 
\[
B_0=\begin{bmatrix}
0& 1 &  &  & \\
& 0 & 1 &  & \\
&  & \ddots & \ddots & \\
&  &  & 0 & 1\\
 &  &  &  & 0
\end{bmatrix}_{(r+1)\times (r+1)}.
\]
Clearly, $B_0^r$ is non-zero, so $B^r$ is non-zero. We then obtain the desired 
$B$, concluding the proof.
\end{proof}

\begin{remark}
Note that $G$ being acyclic implies that the adjacency matrix $A_G$ of $G$ is 
nilpotent. The reverse direction only holds when $A_G$ is defined over 
fields of appropriate characteristics. For instance, let $G$ be the $n$-vertex complete directed graph with self-loops, 
whose adjacency matrix is the $n \times n$ all-one matrix. Then $A_G$ is nilpotent over any field of characteristic dividing $n$, although $G$ is cyclic. On the other hand,~\cref{thm:acyclicity} holds over any field. 
\end{remark}

\subsubsection{Cycle covers and the number of zero eigenvalues}\label{subsec: quantatitative}

Let $G$ be a directed graph of order $n$, and $r\in\{0\}\cup[n]$. We say $G$ 
is \emph{$r$-acyclic} if any collection of vertex-disjoint cycles of $G$ covers at most $n-r$ vertices. Note that (1) $r=0$ corresponds to $G$ having a cycle cover; and (2) $r=n$ corresponds to $G$ being acyclic. The following basic correspondence generalizes~\cref{thm:acyclicity}.
\begin{theorem}\label{thm:eigenval}
    Let $G=([n],E)$ be a directed graph and $\cS_G\leq \M(n,\F)$ be the associated 
    graphical matrix space. Then $G$ is $r$-acyclic if and only if every matrix 
    $B\in\cS_G$ has at least $r$ zero eigenvalues.
\end{theorem}
\begin{proof}
We first show that if every matrix $B\in\cS_G$ has at least $r$ zero 
eigenvalues, then $G$ is $r$-acyclic. If $G$ is not $r$-acyclic, we can find 
disjoint cycles 
$C_1,\dots,C_\ell$ 
    which cover $k\geq n-r+1$ vertices. Let $B\in\cS_G$ be the matrix whose $(i,j)$th 
    entry is $1$ if $(i,j)$ is an edge in any of the cycles $C_1,\dots,C_\ell$ and $0$ 
    elsewhere. Since $C_1,\dots,C_\ell$ are disjoint, $B$ can be written as the direct sum of the adjacency matrices of $C_1,\dots,C_\ell$ (after a suitable relabeling of the vertices). Then $B$ has $(n-k)$ zero eigenvalues, where $n-k\leq r-1$, a 
    contradiction. 
    
    We then show that if $G$ is $r$-acyclic, then every $B\in\cS_G$ has at 
    least $r$ zero eigenvalues. Note that any matrix $B\in\M(n,\F)$ having at least $r$ zero eigenvalues is equivalent to that the characteristic 
    polynomial $P_B(x)$ of $B$ is of the form $x^{r}Q_B(x)$ for some polynomial 
    $Q_B(x)$ of degree $(n-r)$. On the other hand, write the characteristic polynomial of 
    $B$ as
    \[
    P_B(x)=x^n+(-1)E_1(B)x^{n-1}+\cdots+(-1)^{n-1}E_{n-1}(B)x+(-1)^{n}E_{n}(B),
    \]
    where for $k\in[n]$, $E_k(B)$ is the sum of all $k\times k$ principal minors 
    of $B$ (see e.g.~\cite[Eq. (1.2.13)]{horn_johnson_1985}). To show $B$ 
    has at least $r$ 
    zero eigenvalues, we need to prove that $E_{k}(B)=0$ for all $k\geq n-r+1$. 
    
    Fix $k\geq n-r+1$, and let $I\subseteq [n]$ be a set of size $k$. Let 
    $B_I$ be the principal submatrix of $B$ indexed by $I\times I$ and 
    $\cS[I]=\{B_I\mid B\in\cS\}$. Let $G[I]=(I,E[I])$ be the induced subgraph of $G$ on 
    $I$, where $E[I]=\{(i,j)\in E\mid i,j\in I\}$. Observe that $\cS_G[I] =\cS_{G[I]}\leq 
    \M(k,\F)$. Since $G$ is $r$-acyclic, then for any $I$, $G[I]$ cannot be covered by disjoint cycles, implying the bipartite graph corresponding to $G[I]$, i.e., $(I\times I, E[I])$, does not have a perfect matching. By~\cref{thm:edm67}, we have every matrix in $\cS_{G[I]}$ has rank less than $k$. This shows that every matrix $B_I\in\cS_G[I]=\cS_{G[I]}$ has a zero determinant. Further note that 
    \[
    E_k(B)=\sum_{I\subseteq[n],~|I|=k}\det(B_I).
    \]
    This implies that $E_k(B)=0$ for any $k\geq n-r+1$ and any $B\in\cS_G$. Thus we 
    conclude that $B\in\cS_G$ has at least $r$ zero eigenvalues.
\end{proof}

Starting from a directed graph $G=([n],E)$, we can construct a bipartite graph $B(G)=([n]\times [n],E)$. It is not hard to see that $G$ can be covered by disjoint cycles if and only if $B(G)$ has perfect matchings (which we have used in the second half of the proof).~\cref{thm:eigenval} then can be thought of as a directed graph version of~\cref{thm:edm67}. The difference (of the statements) is caused by the different underlying symmetries of directed graphs and bipartite graphs: although $G$ and $B(G)$ have the same set of arcs (edges), we may allow different permutations on the left and right vertices in $B(G)$, while there is only one permutation on the vertices of $G$. Reflecting this fact on the matrix space side, although the directed graph $G$ and the bipartite graph $B(G)$ share the same matrix space $\cS_G$, the symmetries of $G$ and $B(G)$ induce conjugation action and left-right action on $\cS_G$, respectively. Note that the number of zero eigenvalues of $B\in\cS_G$, which corresponds to the algebraic multiplicity of the zero eigenvalue of $B\in\cS_G$, is invariant under conjugation action; and the rank of $B\in\cS_G$, which corresponds to the geometric multiplicity of the zero eigenvalue of $B\in\cS_G$, is invariant under left-right action. This justifies the difference between~\cref{thm:eigenval,thm:edm67}. 

We remark here again that we expect the basic correspondence in~\cref{thm:eigenval} to yield an inherited correspondence, as stated in~\cref{conj:atkinson}. However, at the moment, we are unable to prove this. We believe that a proof technique like that of~\cref{thm:acyclic_subgraph} (proved below) may suffice to prove~\cref{conj:atkinson}, but there are certain difficulties that we are unable to overcome.

\subsection{The inherited correspondence}\label{subsec:acyclic-inherited}

For a directed graph $G=([n],E)$, let $\mas(G)$ be the maximum size of an acyclic spanning
subgraph of $G$ and let $\lns(\cS_G)$
be the largest dimension of a nil subspace of $\cS_G$. 
For any subspace $\cS$ of $\M(n,\F)$, its \emph{supporting (directed) graph} 
$H=([n],E)$ is defined as follows: $(i,j)\in E$ if there is $A\in \cS$ such that 
$A(i,j)\neq 0$. For two matrix spaces $\cS,\cT\leq \M(n,\F)$, $\cS\leq \cT$ implies the
corresponding supporting graph $H_\cS$ is a spanning subgraph of $H_\cT$.

The following notion was introduced by de Seguins Pazzis in~\cite{Paz13}.
\begin{definition}
For a nil matrix space $\cS\leq \M(n,\F)$, we say a non-zero column vector $v$ is 
$\cS$-adapted, if for any $X\in \cS$, $\colspan(X)\neq\langle v\rangle$.  
\end{definition}
Roughly speaking, the notion of an $\cS$-adapted vector is analogous to the notion 
of a sink in a directed graph. Indeed, if vertex $j$ in a directed graph $G$ is a 
sink, then $j$ has no out-neighbors, which means that the $j$th row of the 
adjacency matrix of $G$ contains only zeros. Similarly, if the standard column basis 
vector $e_j$ is $\cS$-adapted, then no non-zero element of $\cS$ is supported in the 
$j$th row of $\cS$.

An important result in \cite{Paz13} is the following. 
Both the statement and the proof can be seen as analogs of the basic fact that 
an acyclic directed graph has a sink.
\begin{lemma}[\hspace{1sp}{\cite[Lemma 5]{Paz13}}]\label{lem: elementary 
basis adapted}
For any nil subspace $\cS\leq \M(n,\F)$, there is some $j\in[n]$ such that the $j$th column standard basis vector $e_j$ is 
$\cS$-adapted.
\end{lemma} 
We utilize this lemma to prove the following:
\begin{lemma}\label{lem: nil to acyclic graph}
For any nil subspace $\cS\leq \M(n,\F)$ with supporting graph $H=([n],F)$, there is 
an acyclic spanning subgraph $H'=([n],F')$ of $H$ with $\dim(\cS)\leq|F'|$.
\end{lemma}
\begin{proof}
The lemma holds for $n=1$. Assume for any nil subspace $\cS\leq \M(n-1,\F)$ with 
supporting graph $H=([n-1],F)$, such $H'$ exists. Consider a nil subspace $\cS\leq 
\M(n,\F)$. By~\cref{lem: elementary basis adapted}, without loss of generality, 
assume $e_n$ is $\cS$-adapted. 
Let $M\in \cS$ be written in the following form
\[
M=\begin{bmatrix}
A(M)&B(M)\\C(M)&D(M)
\end{bmatrix}
\]
where $A(M)\in \M(n-1,\F)$, $B(M)^t,C(M)\in \F^{n-1}$ and $D(M)\in\F$. Denote the 
linear space $A(\cS)=\{A(M):M\in \cS\}$. 
Let $\cT=\{M\in \cS: B(M)=0\}$. Since $\cS$ is nil, we have $D(M)=0$ and $A(M)$ is 
nilpotent for any $M\in \cT$. Consider the linear map $\Phi: \M(n,\F)\to \M(n,\F)$ 
which maps $M$ to $\begin{bmatrix}
0&B(M)\\0&0
\end{bmatrix}$. By the rank-nullity theorem, 
\[
\dim(\cS)=\dim(B(\cS))+\dim(\ker(\Phi)|_\cS)=\dim(B(\cS))+\dim(\cT).
\]
On the other hand, for any matrix $M\in \cT$, if $A(M)=0$, we have $C(M)=D(M)=0$ since 
$e_n$ is $\cS$-adapted. This shows that the projection mapping $M \in \cT$ to $A(M)$ is injective, which implies that $\dim(\cT)=\dim(A(\cT))$. 

By the induction hypothesis, since $A(\cT)\leq \M(n-1,\F)$ is nil, there is an 
acyclic spanning subgraph $H'_{A(\cT)}=([n-1],F'_{A(\cT)})$ of its supporting graph 
$H_{A(\cT)}=([n-1],F_{A(\cT)})$ satisfying $\dim(A(\cT))\leq|F'_{A(\cT)}|$. Note that 
$H_{A(\cT)}$ is an induced subgraph of the supporting graph $H=([n],F)$ of $\cS$. 
Construct $F_\cS$ by adding all possible arcs pointing to $n$ to 
$F'_{A(\cT)}$ and let $H'_\cS=([n],F_\cS)$. Since $H'_{A(\cT)}$ is acyclic and $n$ is a new vertex, $H'_\cS$ is also 
acyclic and $|F_\cS|$ equals $|F'_{A(\cT)}|$ plus the in-degree of $n$ in $H$. Note 
that the latter is lower bounded by $\dim(B(\cS))$. Now we have
\[
|F_\cS|\geq|F'_{A(\cT)}|+\dim(B(\cS))\geq \dim(A(\cT))+\dim(B(\cS))=\dim(\cS).\qedhere
\]
\end{proof}
Now we are ready to prove~\cref{thm:acyclic_subgraph}.

\paragraph{\cref{thm:acyclic_subgraph}, restated.}
\textit{Let $G=([n], E)$ be a directed graph, and $\cS_G\leq\M(n, \F)$ be the matrix space 
associated with $G$. Then $\mas(G)=\lns(\cS_G)$.}
\begin{proof}
It is straightforward to see $\mas(G)\leq\lns(\cS_G)$. Taking a maximum acyclic spanning
subgraph $G'=([n],E')$, we know that $\cS_{G'}$ is nil 
by~\cref{thm:acyclicity} and $\cS_{G'}\leq\cS_G$. Thus 
$\mas(G)=\dim(\cS_{G'})\leq\lns(\cS_G)$.

To see $\lns(\cS_G)\leq\mas(G)$, let $\cS$ be the largest nil subspace of $\cS_G$ with 
supporting graph $H$. By~\cref{lem: nil to acyclic graph}, there is an acyclic 
subgraph $H'$ of $H$ of size at least $\dim(\cS)$. Noting that $H$ is a spanning subgraph of 
$G$, then $H'$ is also a spanning subgraph of $G$. Thus $\dim(\cS)\leq\mas(G)$.
\end{proof}
\begin{remark}\label{remark: Gersten}
\cref{thm:acyclic_subgraph} can be thought of as a generalization of Gerstenhaber's theorem on dimensions of nil matrix spaces~\cite{Ger58}. Gerstenhaber proved that the largest dimension of a nil matrix spaces is $\frac{n(n-1)}{2}$.\footnote{Gerstenhaber's original proof required the underlying field to be sufficient large~\cite{Ger58}, this restriction was removed later by~\cite{zbMATH03952966}. Alternative proofs can be found in, e.g.~\cite{MATHES1991215,MACDONALD20122210,Paz13} and our proof adapts the strategy in~\cite{Paz13}.} Taking $G$ to be the complete directed graph (with self-loops on every vertex), then $\cS_G=\M(n,\F)$. In this case, $\lns(\M(n,\F))$ is given by the size of the maximum acyclic spanning subgraph of $G$, which is $\frac{n(n-1)}{2}$. This recovers Gerstenhaber's theorem.
\end{remark}

\subsection{The induced correspondence}\label{sec: vertex feedback}

For a directed graph $G=([n],E)$, let $\mai(G)$ be the maximum order of an acyclic 
induced subgraph of $G$. 
For $\cS \leq \M(n,\mathbb C)$ and $u \leq \mathbb C^n$, recall the definition of the induced subspace $\cS[U]$ from \cref{def:induced}.
Let $\mrns(\cS)$ be the
largest dimension over $U$ such that $\cS[U]$ is nil.

The following lemma is a reformulation of~\cite[Theorem 4]{MATHES1991215}.

\begin{lemma}[\hspace{1sp}{\cite{MATHES1991215}}]\label{lem: rank-1 nil space}
Let $\cS\leq\M(n, \F)$ be a matrix space spanned by rank-$1$ matrices. Then 
$\cS$ is nilpotent if and only if $\cS$ is nil.
\end{lemma}

\paragraph{\cref{thm:acyclic_restriction}, restated.} \textit{Let $G=([n], E)$ be a 
directed graph, and $\cS_G\leq\M(n, \mathbb{C})$ be the associated graphical 
matrix space. Then $\mai(G)=\mrns(\cS_G)$.}

\begin{proof}
To see $\mai(G)\leq\mrns(\cS_G)$, let $S\subseteq[n]$ be a vertex subset of the 
maximum order such that the induced subgraph $G[S]$ is acyclic. Let $U=\langle e_i\mid i\in S\rangle$. 
Then $\cS_G[U]$ is nil, as 
$\cS_G[U]=\cS_{G[S]}$ and $G[S]$ is acyclic.

To see $\mrns(\cS_G)\leq\mai(G)$, let $U\leq 
\mathbb{C}^n$ be a subspace of the largest dimension such that $\cS_G[U]$ is nil. 
Supposing  
$\dim(U)=k$, 
let $T_U\in \M(k\times n, \mathbb{C})$ be a matrix whose rows form an 
orthonormal basis 
of $U$. For $i\in[n]$, let $r_i$ be the $i$th column of 
$T_U$. Then $\cS_G[U]=\langle r_ir_j^*\mid \ (i,j)\in E\rangle$. In particular, 
$\cS_G[U]$ is spanned by rank-$1$ matrices.

Let $S\subseteq[n]$, 
$|S|=k$, such that $\{r_i\mid i\in S\}$ is a set of linearly independent column 
vectors. In particular, for $i\in S$, $r_i$ is 
non-zero. 
We claim that $G[S]$ is acyclic. If not, after a possible relabeling of vertices, 
let $(1, \dots, \ell,1)$ be a cycle in 
$G[S]$. Then for any $j\in[\ell-1]$, 
$r_{j}r_{j+1}^*$ is in $\cS_G[U]$ and $r_{\ell}r_{1}^*\in\cS_G[U]$. 
Note that
\[
(r_{1} r_{2}^*)(r_{2} r_{3}^*)\cdots 
(r_{\ell}r_{1}^*)=r_{1}(r_{2}^*r_{2})\dots(r_{\ell}^*r_{\ell})r_{1}^*
=\alpha r_{1}r_{1}^*,
\]
where $\alpha=(r_{2}r_{2}^*)\dots(r_{\ell}r_{\ell}^*)\neq 0$. 
This implies that $(r_{1} r_{2}^*)(r_{2} r_{3}^*)\cdots (r_{\ell} r_{1}^*)$ is not 
nilpotent, so $\cS[U]$ is not nilpotent. As $\cS[U]$ is spanned by rank-$1$ matrices, 
by~\cref{lem: 
rank-1 nil space}, $\cS[U]$ is not nil. This is a contradiction to our assumption, 
concluding the proof. 
\end{proof}

\section{Strong connectivity and irreducibility} \label{sec:conn}
In this section, we establish the correspondences between strong connectivity and irreducibility. We begin with the basic correspondences, then prove the inherited correspondence in~\cref{subsec: arc strong connectivity}, and finally the induced correspondence in~\cref{sec: vertex strong connectivity}.

\subsection{The basic correspondences}

Let $G=([n], E)$ be a directed graph. A set of vertices $V\subseteq [n]$ is 
\emph{invariant} in $G$, if there is no arc from $V$ to $[n]\setminus 
V$. Recall that $G$ is \emph{strongly connected}, if the only non-empty invariant vertex set 
is $[n]$ itself. 
A strongly connected component decomposition is a partition $[n]=V_1\cup \dots \cup V_k$, such that for any $i\in[k]$, $G[V_i]$ 
is a maximal strongly connected induced subgraph. It is clear that for a graph 
$G$, there exists a unique strongly connected component decomposition by choosing $V_1 <_t V_2 <_t \dots <_t V_k$, where $<_t$ is the topological ordering of the strongly connected components of $G$. Let $c(G)$ be the number of strongly connected components in $G$. 

Let $\cS\leq\M(n, \F)$ be a matrix space. A subspace $U\leq \F^n$ is \emph{invariant} in $\cS$ if $\cS(U):=\langle \cup_{B\in \cS}B(U)\rangle\leq U$. If $\cS$ admits a nontrivial invariant subspace (i.e., not $\{0\}$ and $\F^n$), we say $\cS$ is \emph{reducible}; otherwise, we say $\cS$ is \emph{irreducible}.
Let $0=U_0< U_1< \cdots < U_k=\F^n$ be a chain of invariant subspaces of $\cS$. This chain is maximal if for 
any $i\in[k]$, the induced matrix space of $\cS$ on $U_i/U_{i-1}$ is irreducible. 
By the Jordan--H\"older theorem (see e.g.~\cite[Theorem 3.5]{lang2005algebra}), two maximal chains of invariant subspaces of $\cS$ 
are of the same length. Let $c(\cS)$ be the length of a maximal chain of invariant 
subspaces. 

\begin{theorem}\label{thm:conn_s}
Let $G=([n], E)$ be a directed graph, and $\cS_G\leq\M(n, \F)$ be the graphical 
matrix space associated with $G$. Then 
$c(G)=c(\cS_G)$. In particular, $G$ is strongly 
connected if and only if $\cS_G$ is irreducible.
\end{theorem}
\begin{proof}
We begin by proving the special case, that $G$ is strongly connected if and only if $\cS_G$ is irreducible (cf.~\cref{thm:conn}). 
This will then be used in the proof of the more general result that $c(G)=c(\cS_G)$.

If $G$ is not strongly connected, then there is a non-empty subset $V\subset [n]$ which is invariant in $G$. Without loss of generality assume $V=[d]$ for some $d\in[n-1]$. Consider the subspace $U=\langle e_1,\dots,e_d\rangle\leq \F^n$. It is straightforward to verify that $U$ is an invariant subspace. Thus $\cS_G$ is not irreducible.

If $\cS_G$ is reducible, then there exists a subspace $U\leq\F^n$ of 
dimension $d$, $d\in[n-1]$, such that for any $B\in\cS_G$, $B(U)\leq U$. Let 
$T_U\in \M(d\times n,\F)$ be a matrix whose rows span $U$. 
As $\rk(U)=d$, there exist a permutation matrix $P\in\GL(n,\F)$ and some invertible matrix
$T\in\GL(d,\F)$ such that $\hat{U}=TUP$ is of the form $\begin{bmatrix} 
I_d&U_0\end{bmatrix}$ for some matrix $U_0\in \M(d\times (n-d),\F)$. Then the 
row span of $\hat{U}$ is invariant under the action of $P\cS_G 
P^{-1}=\cS_{\hat{G}}$, 
where $\hat{G}$ is isomorphic to $G$ with respect to the permutation $P$. 

We claim that $\hat{G}$ is not strongly connected, which implies that $G$ is not 
strongly connected either. 
To see this, we show that $[d]$ is invariant, namely there is no 
arc from $[d]$ to $[n]\setminus [d]$. By way of contradiction, suppose there exists 
$(i,j)\in E(\hat{G})$ with 
$i\in[d]$ and $j\in[n]\setminus[d]$. Thus $\E_{i,j}\in\cS_{\hat{G}}$. Note 
that the $i$th row vector of $\hat{U}$ is 
$\hat{u}_i=e_i+(0,u_{0,i})\in\rowspan(\hat{U})$, where $u_{0,i}\in\F^{n-d}$.
The image of $\hat{u}_i$ under $\E_{i, j}$ is $\hat{u}\E_{i,j}=e_j$. As $\rowspan(\hat{U})$ is invariant under $\cS_{\hat{G}}$, we have 
$e_j\in\rowspan(\hat{U})$. This is impossible, because a non-zero vector in 
$\rowspan(\hat{U})$ has at least one non-zero entry in the first $d$ 
coordinates. To summarize, there is no arc from $[d]$ to $[n]\setminus [d]$, 
which shows that $\hat{G}$ is not strongly connected.

To see $c(G)=c(\cS_G)$, we first show that $c(G)\leq c(\cS_G)$. Suppose $[n]=V_1\cup\dots\cup V_k$ is the strongly connected component decomposition of $G$, where $V_1 <_t V_2 <_t \dots <_t V_k$. This naturally 
gives rise to a chain of subspaces $0=U_0< U_1< U_2<\cdots< U_k=\F^n$ 
where $U_i=\langle \{e_i \mid i\in V_1\cup\dots\cup V_i\}\rangle$. It is straightforward to verify that this gives rise to a chain of invariant 
subspaces of $\cS_G$. 

We then show that $c(G)\geq c(\cS_G)$. Let $0=U_0< U_1< U_2< \cdots< 
U_k=\F^n$ be a maximal chain of invariant subspaces of $\cS_G$. Suppose 
$\dim(U_i)=d_i$. Let $T\in \GL(n, \F)$ be a matrix whose first $d_i$ columns span 
$U_i$. By a repeated use of Laplace expansion (starting from the last 
$d_k-d_{k-1}$ columns), there exists $\emptyset=V_0\subset V_1 \subset V_2\subset 
\dots \subset V_k=[n]$, such that for any $i\in[k]$, the submatrix of $T$ with row 
indices $V_i$ and column indices $[d_i]$ is full-rank. By the argument above, 
there are no arc in $G$ that go from $ V_i$ to $[n]\setminus V_i$. It follows 
that for any $i\in[k]$, $V_i\setminus V_{i-1}$ is invariant, and $c(G)\geq 
c(\cS_G)$.
\end{proof}

\subsection{The inherited correspondence}\label{subsec: arc strong connectivity}
For a directed graph $G=([n],E)$, let $\Nsc(G)$ be the maximum number of arcs in a 
non-strongly-connected subgraph of $G$. Note that 
$\lambda(G):=|E|-\Nsc(G)$ is known as the \emph{arc-strong 
connectivity} of $G$ (see e.g.~\cite[Chap. 1.5]{BG08}), which is the minimum 
number of arcs of $G$ whose removal makes $G$ not strongly connected. 

For a matrix space $\cS\leq \M(n,\F)$, let $\Rdc(\cS)$ be the 
largest dimension of a reducible subspace of $\cS$. 
We are ready to establish the following result:

\paragraph{\cref{thm:lambda}, restated.} \textit{Let $G=([n], E)$ be a directed graph, and $\cS_G\leq\M(n, \F)$ be the matrix space associated with $G$. Then $\Nsc(G)=\Rdc(\cS_G)$.}

\begin{proof}
To see that $\Nsc(G)\leq\Rdc(\cS_G)$, take a maximum spanning
subgraph $G'$ of $G$ which is not strongly connected. Then $\cS_{G'}\leq \cS_G$ 
and the reducibility of $\cS_{G'}$ follows from~\cref{thm:conn_s}.

To show that $\Rdc(\cS_G)\leq\Nsc(G)$, suppose that $\cC\leq\cS_G$ admits a 
non-trivial and proper invariant subspace $U\leq\F^n$, where
$\dim(\cC)=\Rdc(\cS_G)$, $\dim(\cS_G)=m$, and $\dim(U)=d$. Let $T\in \GL(n, \F)$ be an invertible matrix such that $T=\begin{bmatrix}
T_U\\
T_V
\end{bmatrix}$, where $T_U\in \M(d\times n, \F)$ (resp.\ $T_V\in \M((n-d)\times n, \F)$) whose rows span $U$ (resp.\ a complementary subspace $V$). Let $T^{-1}=\begin{bmatrix}
L & R 
\end{bmatrix}$ where $L\in \M(n\times d, \F)$ and $R\in\M(n\times (n-d),\F)$. 

Note that for $B\in\M(n, \F)$, $TBT^{-1}=\begin{bmatrix}
T_UBL & T_UBR\\
T_VBL & T_VBR
\end{bmatrix}$. 
We are interested in $\cD:=T_U\cS_G R\leq \M(d\times (n-d), \F)$. Since $U$ is an 
invariant subspace of $\cC\leq \cS_G$, for any $C\in \cC$, $T_UCR=0\in\M(d\times (n-d), \F)$. Therefore, for any $C\in \cC$, $TCT^{-1}$ is of the form $\begin{bmatrix}
C_1& 0 \\
C_2& C_3
\end{bmatrix}$ where $C_1\in\M(d\times d, \F)$. It follows from the rank-nullity theorem that $\dim(\cD)\leq m - \dim(\cC)= m - \Rdc(\cS_G)$. In fact, we have $\dim(\cD)=m - \Rdc(\cS_G)$, as otherwise there exists another reducible $\cC'\leq \cS_G$ of dimension greater than $\Rdc(\cS_G)$ as witnessed by $U$.

Our goal is to construct a set of arcs in $G$ of size at most $m-\Rdc(\cS_G)$, whose removal makes $G$ not strongly connected. To achieve this, we examine 
$T=\begin{bmatrix}
T_U\\
T_V
\end{bmatrix}$. Since $T_U\in\M(d\times n, \F)$ is of rank $d$, there exists a 
$d\times d$ full-rank submatrix in $T_U$. Therefore, by a permutation of columns of $T$ and relabeling the vertices of $G$ if necessary, we can put $T$ in the form $
\begin{bmatrix}
T_{1,1} & T_{1,2} \\
T_{2,1} & T_{2,2}
\end{bmatrix}$ where $T_{1,1}\in\GL(d, \F)$. Then $T^{-1}$ is of the form $
\begin{bmatrix}
S_{1,1} & S_{1,2} \\
S_{2,1} & S_{2,2}
\end{bmatrix}$, where $S_{2,2}=T_{2,2}-T_{2,1}T_{1,1}^{-1}T_{1,2}$ is known as the 
Schur complement of $T_{1,1}$~\cite{Zha06}. The basic observation that $S_{2,2}$ is 
invertible will be crucial in the following. 
 
We claim that in $G=([n], E)$ (after a suitable relabeling), the number of arcs going from $[d]$ to $[n]\setminus [d]$ is 
no more than $\dim(\cD)=m - \Rdc(\cS_G)$. Let $T_{1,1}=\begin{bmatrix}
c_1^t & \dots & c_d^t
\end{bmatrix}$ where $c_i\in \F^d$, and $S_{2,2}=\begin{bmatrix}
r_1 \\
\vdots \\
r_{n-d}
\end{bmatrix}$ where $r_i\in \F^{n-d}$. Since $T_{1,1}$ and $S_{2,2}$ are 
invertible, $c_i$'s (resp.\ $r_j$'s) are 
linearly independent. It follows that $\{T_U\E_{i,j}R \mid i\in[d], 
j\in[n]\setminus [d]\}\subseteq\M(d\times (n-d), \F)$ are linearly independent, as 
$T_U\E_{i,j}R=\begin{bmatrix}
T_{1,1} & T_{1,2}
\end{bmatrix}\E_{i,j}\begin{bmatrix}
S_{1,2}\\ S_{2,2}
\end{bmatrix}=c_i^t r_{j-d}\in \M(d\times (n-d), \F)$. Let the arcs in $G=([n], 
E)$ from $[d]$ to $[n]\setminus [d]$ be $(i_1, j_1), \dots, (i_\ell, j_\ell)$, where
$i_1,\dots,i_\ell\in[d]$ and $j_1,\dots,j_\ell\in[n]\setminus [d]$. Then $\cD$ contains $T_U\E_{i_k, j_k}R$, 
$k\in[\ell]$, which are linearly independent. It follows that $\ell\leq 
\dim(\cD)=m-\Rdc(\cS_G)$. That is, by removing $\ell\leq m-\Rdc(\cS_G)$ 
arcs, $G$ becomes not strongly connected. This concludes the proof. 
\end{proof}

\subsection{The induced correspondence}\label{sec: vertex strong connectivity}
For a directed graph $G=([n],E)$, let $\Insc(G)$ be the maximum order of a 
non-strongly-connected induced subgraph of $G$. Note that 
$\kappa(G):= n-\Insc(G)$ is known as the \emph{vertex-strong 
connectivity} of $G$ (see e.g.~\cite[Chap. 1.5]{BG08}), which is the minimum 
number of vertices of $G$ whose removal makes $G$ not strongly connected. 

For a matrix space $\cS\leq \M(n,\mathbb{C})$, let $\Irdc(\cS)$ be the largest dimension of $U\leq\mathbb{C}^n$ such that $\cS[U]$ is reducible. We are ready to establish the following result:

\paragraph{\cref{thm:kappa}, restated.} \textit{Let $G=([n], E)$ be a directed graph, and $\cS_G\leq\M(n, \mathbb{C})$ be the matrix space associated with $G$. Then $\Insc(G)=\Irdc(\cS_G)$.}

\begin{proof}
To see that $\Insc(G)\leq\Irdc(\cS_G)$, take a maximum induced
subgraph $G'=(V(G'),E(G'))$ of $G$ which is not strongly connected. Let $U$ be the 
subspace spanned by $\{e_i\mid i\in V(G')\}$. Then by~\cref{thm:conn_s}, $\cS_G[U]\leq \M(|V(G')|,\mathbb{C})$ is 
reducible since $G'$ is not strongly connected.

To see that $\Insc(\cS_G)\leq\Irdc(G)$, take a subspace $U$ of dimension $d$ 
such that $\cS_G[U]$ admits a nontrivial invariant subspace $V\leq U$. Let $T_U\in \M(d\times n,\mathbb{C})$ be the matrix whose rows are an orthonormal basis of $U$. Take a unitary matrix
 $R\in \M(d,\mathbb{C})$ which maps the first $b=\dim(V)$ orthonormal basis vectors of $U$ to an orthonormal basis of $V$. Then for any matrix 
$B\in\cS_G$, 
\begin{equation}\label{eq: block-upper-triangular}
RT_UBT_U^*R^*=\begin{bmatrix}B_{1,1}&0\\B_{2,1}&B_{2,2}\end{bmatrix}
\end{equation}
for some $B_{1,1}\in \M(b,\mathbb{C})$, $B_{2,1}\in \M(b\times (d-b),\mathbb{C})$ and 
$B_{2,2}\in   
\M(d-b,\mathbb{C})$. 
Without loss of generality, assume the first $d$ columns of $RT_U\in \M(d\times 
n,\mathbb{C})$ are linearly independent and denote this submatrix as $T\in\GL(d,\mathbb{C})$. (Otherwise, we 
can permute the columns of $T_U$ and permute the vertices of $G$ accordingly). 
Moreover, there is a partition $[d]=I_1\cup I_2$ with $|I_1|=b$, such that the 
submatrix $T_1$ (resp.\ $T_2$) of $T$ with column indices from $I_1$ (resp.\ $I_2$) 
and row indices from $[b]$ (resp.\ $[n]\setminus[b]$) is invertible.

We shall prove that there is no edge of $G[[d]]$ going from $I_1$ to $I_2$. 
Suppose for contradiction that we have $(i,j)\in E$ for some $i\in I_1$ and $j\in I_2$. Let the 
$i$th column of $RT_U$ be $t_i$. Then 
$RT_U\E_{i,j}T_U^*R^*=t_it_j^*$ is of the form as in~\cref{eq: 
block-upper-triangular}. Thus, the $k$th coordinate of $t_{i}$ is $0$ for all 
$k\in[b]$ or the $\ell$th coordinate of $t_{j}$ is $0$ for all 
$\ell\in[d]\setminus[b]$. In the first case, we have the $i$th column of $T_1$ is $0$, 
contradicting $T_1$ being invertible; and in the second case, we have the $j$th 
column of $T_2$ is $0$, contradicting $T_2$ being invertible. Thus there cannot be an edge from $I_1$ to $I_2$, which proves that $G[[d]]$ is not strongly connected, and $\Insc(\cS_G)\leq\Irdc(G)$ follows.
\end{proof}

\section{Isomorphism, congruence, and conjugacy}\label{sec:iso}

Recall that two directed graphs $G=([n], E)$ and $H=([n], F)$ are isomorphic, if 
there exists a bijective map $f:[n]\to [n]$, such that $(i, j)\in E$ if and only 
if $(f(i), f(j))\in F$. In this section we extend this notion to graphical matrix 
spaces.
Recall that two matrix spaces $\cS, \cC\leq \M(n, 
\F)$ are conjugate (resp.\ congruent) if and only if there exists $T\in \GL(n, 
\F)$, such that $\cS=T\cC T^{-1}$ (resp.\ $\cS=T\cC T^{t}$). When $\F=\mathbb{C}$, $\cS, \cC\leq \M(n, 
\F)$ are congruent if and only if there exists $T\in \GL(n, 
\F)$, such that $\cS=T\cC T^{*}$. 

Note that the action of $\S_n$, when considered as a subgroup of $\GL(n,\F)$, is 
compatible with both the conjugacy and congruence actions of $\GL(n,\F)$. 
This immediately gives us the following. 
\begin{observation}
Let $G=([n], E)$ and $H=([n], F)$ be two directed graphs and $\cS_G$ and $\cS_H$ be the associated graphical matrix spaces, respectively. If $G$ is isomorphic to a subgraph of $H$, then $\cS_G$ is conjugate to (resp.\ congruent to) a subspace of $\cS_H$.
\end{observation}

\subsection{Isomorphism and congruence}

We show that graphical matrix space congruence implies graph isomorphism by 
proving a stronger result: $G$ is isomorphic to a subgraph of $H$ if and only if 
$\cS_G$ is congruent to a subspace of $\cS_H$.
\begin{proposition}\label{prop:congruence}
Let $G=([n], E)$ and $H=([n], F)$ be directed graphs. Let $\cS_G$ and $\cS_H$ be 
the associated graphical matrix spaces, respectively. If $\cS_G$ is congruent to a subspace of $\cS_H$, then $G$ is isomorphic to a subgraph of $H$.
\end{proposition}
\begin{proof}
Let $T\in\GL(n,\F)$ such that $T\cS_GT^t\leq \cS_H$. Let $t_{i,j}\in\F$ be the $(i,j)$th entry of $T$. Then for any $(i,j)\in E$, we have
\[
T\E_{i,j}T^t
=\begin{bmatrix}t_{1,i}\\\vdots\\t_{n,i}\end{bmatrix}\begin{bmatrix}t_{1,j}&\cdots&t_{n,j}\end{bmatrix}=[t_{k,i}t_{\ell,j}]_{k,\ell\in[n]}\in\cS_H.
\]
By the structure of $\cS_H$, if the $(k,\ell)$th entry $t_{k,i}t_{\ell,j}$ of 
$T\E_{i,j}T^t$ is 
non-zero, then $(k,\ell)\in F$. As $T$ is 
invertible, there exists a permutation $\sigma:[n]\to[n]$ such that 
$t_{\sigma(i),i}\neq 0$ for any $i\in[n]$. 
In particular, $t_{\sigma(i),i}t_{\sigma(j),j}\neq 0$ for 
any $i,j\in[n]$. Thus, we have that for $(i,j)\in E$, 
$(\sigma(i),\sigma(j))\in F$. In other words, $\sigma$ is an injective map from 
vertices of $G$ to vertices of $H$ which preserves arcs, i.e., an embedding of $G$ 
into $H$. This shows that $G$ is a subgraph of $H$, as claimed.
\end{proof}

\begin{remark}
	Recall that when $\F=\mathbb{C}$, we define congruence using $T^*$ rather than 
	$T^t$. It can be verified that the above argument works verbatim in case 
	$\F=\mathbb{C}$, by simply replacing every instance of a transpose by a 
	conjugate transpose and by replacing $t_{\ell,j}$ by $\overline{t_{\ell,j}}$.
\end{remark}
The proof of \cref{prop:congruence} yields the following 
corollary which will be useful in \cref{subsec:transitivity}.
\begin{corollary}\label{cor:cong_conseq}
	Let $G=([n], E)$ and $H=([n], F)$ be directed graphs with associated graphical matrix spaces $\mathcal{S}_G$ and 
	$\mathcal{S}_H$, respectively. If $T\mathcal{S}_GT^{t}=\mathcal{S}_H$ for some $T\in\GL(n,\F)$, then any permutation $\sigma$ satisfying that the $(i,\sigma(i))$th entry of $T$ is non-zero for each $i\in[n]$ gives an isomorphism between $G$ and $H$.
\end{corollary}

The next natural thing to examine is whether this basic correspondence can be 
boosted to an inherited correspondence. Namely, is the maximum size over subgraphs 
of $G$ that do not contain an isomorphic copy of $H$ equals the largest dimension 
over subspaces of $\cS_G$ that do not contain a congruent copy of $\cS_H$? 
Interestingly, this is not true in general, as seen in the following example.
\begin{example}\label{example: no inherited version of congruence}
	Let $G=([2],\{(1,2),(2,1)\})$ and $H=([2],\{(1,2)\})$. Then any non-empty 
	subgraph of $G$ contains an isomorphic copy of $H$, thus the maximum size over 
	subgraphs of $G$ with no isomorphic copy of $H$ is $0$. On the other hand, 
	take the one-dimensional subspace 
	$\cS=\left\{\begin{bmatrix}0&x\\x&0\end{bmatrix}:~x\in\F\right\}$ of 
	$\cS_G$. As congruence preserves symmetry, $\cS_H$ cannot be a subspace of any 
	congruent copy of $\cS$. 
	Then the largest dimension over subspaces of $\cS_G$ that do not contain a 
	congruent copy of $\cS_H$ is at least $1$. 
\end{example}

\subsection{Isomorphism and conjugacy}
We now show that graphical matrix space conjugacy implies graph isomorphism. The proof is inspired by~\cite[Theorem 4.13 in arXiv version 2]{BS20} and is much more complicated than that of~\cref{prop:congruence}. To prepare for its proof, we need the following notation.

A $3$-way array is a rectangular cuboid of field elements. Let $\T(\ell\times 
n\times m, \F)$ be the linear space 
of $\ell\times n\times m$ 3-way arrays over $\F$. We use the fixed-width 
teletypefont for 3-way 
arrays, such as $\tA$, $\tens{B}$, and so on.

Given $\tA\in \T(\ell\times n\times m, \F)$, the $(i,j,k)$th entry of $\tA$ is 
denoted as $\tA(i,j,k)\in \F$. We can slice $\tA$ along one direction and obtain 
several 
matrices, which are called slices. For example, slicing along the first 
coordinate, we obtain the \emph{horizontal} slices, namely $\ell$ matrices $A_1, 
\dots, A_\ell\in \M(n\times m, \F)$, where $A_i(j,k)=\tA(i,j,k)$. Similarly, we 
also obtain the \emph{vertical} slices by slicing along the second coordinate, and 
the \emph{frontal} slices by slicing along the third coordinate. 

A $3$-way array allows for general linear group actions in three directions. 
Given $P\in \M(\ell, 
\F)$ and $Q\in \M(n, \F)$, let 
$P\tA Q$ be the $\ell \times 
n\times m$ $3$-way array whose $k$th frontal slice is $P A_k Q$. For 
$R=(r_{i,j})\in 
\GL(m, \F)$, let $\tA^R$ be the $\ell\times n\times m$ $3$-way array whose 
$k$th 
frontal slice is $\sum_{k'\in[m]}r_{k',k}A_{k'}$.  

Given $\tA\in \T(\ell\times n\times m, \F)$, sometimes it is convenient to work 
with the associated 3-way arrays obtained by permuting the three indices. For 
example, from $\tA$ above we can construct a 3-way array $\tens{B}\in \T(n\times 
m\times \ell, \F)$, such that the $i$th frontal slice of $\tens{B}$ is the $i$th 
horizontal slice of $\tA$. 

Let $\cA, \cB\leq\M(n, \F)$ and suppose $\dim(\cA)=\dim(\cB)=m$. We can form 
$3$-way arrays $\tA$ and $\tB$ in $\T(n\times n\times m, \F)$, by taking ordered 
bases of $\cA$ and $\cB$, respectively. Then $\cA$ and $\cB$ are conjugate, if and 
only if there exist $T\in\GL(n, \F)$ and $R\in\GL(m, \F)$, such that 
$T\tA T^{-1}=\tB^R$.

We now turn to graphical matrix spaces, and deduce a key lemma 
(\cref{lem:conj_key}) that characterizes conjugacy of graphical matrix 
spaces. The following notation will be used several times in slightly different contexts, 
so we collect them as a definition for future references. 

\begin{definition}\label{def:conj}
Let $G=([n], E)$ and $H=([n], F)$ be 
directed graphs of size $m$. Let $\cS_G$ and $\cS_H$ be the associated 
graphical matrix spaces, respectively.
Construct $\tS_G\in \T(n\times n\times m, \F)$ by setting its frontal slices as $(\E_{i, j} \mid (i, 
j)\in E)$, where edges in $G$ are ordered lexicographically.
Similarly construct $\tS_H\in\T(n\times n\times m, \F)$. 
Let $\tC\in \T(n\times m\times n, \F)$ be the 3-way array whose frontal slices are 
the horizontal slices of $\tS_G$. That is, $\tC(i, j, k)=\tS_G(k, i, j)$. 
Similarly, let $\tD\in \T(n\times m\times n, \F)$ be the 3-way array whose frontal 
slices are the horizontal slices of $\tS_H$. 
Let $C_i$'s 
be the frontal slices of $\tC$, and $D_i$'s the frontal slices of $\tD$.
\end{definition}

\begin{observation}\label{obs:non-overlapping}
Let $C_i$ and $D_i$ be as in~\cref{def:conj}. The non-zero columns of 
$C_i$'s are non-overlapping. That is, for any 
$k\in[m]$, there exists one and only one $i\in[n]$, such that the $k$th column of 
$C_i$ is non-zero. 
The above statement holds for $D_i$'s too. 
In particular, for any $T\in\GL(n, \F)$, the non-zero columns of 
$TD_i$'s are also non-overlapping. 
\end{observation}
\begin{proof}
Suppose the $k$th edge by lexicographic order in $G$ is $(i, j)$. So the $k$th 
frontal slice of $\tS_G$ is $\E_{i,j}$. The $k$th 
column of $C_i$ is the transpose of the $i$th row of $\E_{i,j}$, namely $e_j$. For 
any $i'\neq i$, the $k$th 
column of $C_{i'}$ is the transpose of the $i'$th row of $\E_{i,j}$, which is the 
all-zero vector. The same argument works for $D_i$, and therefore $TD_i$ too.
\end{proof}

An important consequence of~\cref{obs:non-overlapping} is the following.
\begin{lemma}\label{lem:conj_key}
Let $\cS_G, \cS_H\leq\M(n, \F)$ with $\dim(\cS_G)=\dim(\cS_H)=m$, $\tS_G, 
\tS_H\in\T(n\times n\times m, \F)$, and 
$\tC, \tD\in \T(n\times m\times n, \F)$ be in~\cref{def:conj}. Then 
$\cS_G$ and $\cS_H$ are conjugate, if and only if there exists $T\in 
\GL(n, \F)$, such that the following holds: 
letting $C_i'$ be the $i$th 
frontal slice of $\tC^{T}$, then for any $i\in[n]$, 
$\colspan(C_i')\subseteq \colspan(TD_i)$.
\end{lemma}
\begin{proof}
For the only if direction, note that $\cS_G$ and $\cS_H$ being conjugate implies 
that there exist $T\in\GL(n, \F)$ and $R\in\GL(m, \F)$, such that 
$T^t\tS_G=\tS_H^RT^t$. (The use of $T^t$ is to make some notation in the following 
easier.) This translates to $\tC^T=T\tD R$. Therefore, for any $i\in[n]$, 
$C_i'=TD_iR$, which implies that $\colspan(C_i')\subseteq \colspan(TD_i)$.

For the if direction, by~\cref{obs:non-overlapping}, the non-zero 
columns of $TD_i$ over 
$i\in[n]$ are non-overlapping. Let $[m]=S_1\cup\dots\cup S_n$ be a partition such 
that the indices of non-zero columns in $TD_i$ are exactly in $S_i$. Suppose 
$TD_i=[u_1, \dots, u_m]$, where 
$u_i\in\F^n$. Let
$R=(r_{i,j})_{i,j\in[m]}\in\M(m, \F)$. 
Then the $k$th column of $TD_iR$ is $\sum_{j\in [m]}r_{j,k}u_j=\sum_{j\in 
S_i}r_{j,k}u_j$, where the equality is due to the fact that $u_{j}=0$ for $j\not\in 
S_i$. 
In other words, for any $i\in[n]$, only the columns of $R$ whose indices in $S_i$ are effective for the columns $TD_i$. Using the condition that
$\colspan(C_i')\subseteq \colspan(TD_i)$, we can set the rows of $R\in\M(m, 
\F)$ with indices in $S_i$, such that $C_i'=TD_iR$, for each $i\in[n]$. That is,  
$\tC^{T}=T\tD R$, which implies that $T^t\tS_G=\tS_H^{R}T^t$. Now we recall 
that the frontal slices of $\tS_G$ and $\tS_H$ are bases of $\cS_G$ and $\cS_H$ 
which are both $m$-dimensional. It follows that $R$ must be invertible, and 
$T^t\cS_G T^{-t}=\cS_H$.
\end{proof}

We are now ready to show that conjugacy implies isomorphism. 

\begin{proposition}\label{prop:conjugate}
Let $G=([n], E)$ and $H=([n], F)$ be directed graphs of size $m$. Let $\cS_G$ and 
$\cS_H$ be the associated graphical matrix spaces, respectively. If $\cS_G$ and $\cS_H$ are conjugate, then
$G$ and $H$ are isomorphic.
\end{proposition}
\begin{proof}
From $\cS_G$ and $\cS_H$, we construct $\tS_G, \tS_H\in\T(n\times n\times m, \F)$ 
and $\tC, \tD\in\T(n\times m\times n, \F)$ as in~\cref{def:conj}.
Since $\cS_G$ and $\cS_H$ are 
conjugate, there exist $T\in\GL(n, \F)$ and $R\in \GL(m, \F)$ such that 
$T^t\tS_G=\tS_H^RT^t$, which yields 
$\tC^{T}=T\tD R$. 

Denote $T=\begin{bmatrix}
t_1 & t_2 &\cdots & t_n
\end{bmatrix}$, where $t_i=(t_{i,1}, \cdots, t_{i,n})^t$ is the $i$th 
column. By the Laplace expansion, there exists $i\in[n]$, such that $t_{1,i}\neq 
0$ and 
$T'=\begin{bmatrix}
e_i & t_2 & \cdots & t_n
\end{bmatrix}$ is also invertible. 
\begin{claim}\label{claim:conj_technical}
Let $C_i'$ be the $i$th frontal slice of $\tC^{T'}$. Then for any $i\in[n]$, 
$\colspan(C_i')\subseteq \colspan(T'D_i)$. 
\end{claim}
Once~\cref{claim:conj_technical} holds, it would follow that 
$T'^t\cS_GT'^{-t}=\cS_H$  by~\cref{lem:conj_key}. We could then apply this procedure to $T'$ to set $t_2$ 
to be a standard basis vector, and so on, until we get a permutation matrix $P$ 
such that $P^{-1}\cS_GP=\cS_H$. This would allow us to conclude that $G$ and $H$ 
are isomorphic. 

We now prove~\cref{claim:conj_technical}.
\begin{proof}[{Proof of~\cref{claim:conj_technical}.}]
As the difference between $T$ and $T'$ lies on the first column, changing 
from $T$ to $T'$ has the following possible consequences: 
\begin{enumerate}
\item For $\tC^T$ and $\tC^{T'}$, their first frontal slices are different. 
\item For those frontal slices $D_j$ containing $e_1$ as 
a column, the $j$th frontal slices of $T\tD$ and $T'\tD$ are different.
\end{enumerate}
We now distinguish the following cases. 
\begin{enumerate}
\item[(a)] Suppose $e_1\notin\colset(D_j)$ for any $j\geq 2$. This is when neither change happens. Consider $\sum_{k\in[n]}t_{k, j}C_k$, which is equal to $TD_jR$. It is clear that $\colspan(C_j')=\colspan(\sum_{k\in[n]}t_{k, j}C_k)=\colspan(TD_j)=\colspan(T'D_j)$.

\item[(b)] Suppose $e_1\in\colset(D_j)$ for some $j\geq 2$. Then $t_1\in 
\colset(TD_j)$. 
Consider $\sum_{k\in[n]}t_{k, j}C_k$, which is equal to $TD_jR$. That is, 
$\colspan(\sum_{k\in[n]}t_{k, j}C_k)=\colspan(TD_j)$. By~\cref{obs:non-overlapping}, the columns of $\sum_{k\in[n]}t_{k, j}C_k$ 
are scaled standard basis vectors. As the $i$th entry of $t_1\in\colspan(\sum_{k\in[n]}t_{k, j}C_k)$ is non-zero, it follows 
that the $i$th standard basis vector $e_i$ is in $\colset(\sum_{k\in[n]}t_{k, 
j}C_k)$. It follows that $\colspan(\sum_{k\in[n]}t_{k, j}C_k)\supseteq 
\colspan(T'D_j)$. Since $T$ and $T'$ are both full-rank, we have 
$\dim(\colspan(\sum_{k\in[n]}t_{k, 
j}C_k))=\dim(\colspan(TD_j))=\dim(\colspan(T'D_j))$. We then have 
$\colspan(C_j')=\colspan(\sum_{k\in[n]}t_{k, j}C_k)= \colspan(T'D_j)$, where the 
first equality is due to the assumption that $j\geq 2$.

\item[(c)] Suppose $e_1\not\in \colset(D_1)$. The first frontal slice of $\tC^T$ 
is $\sum_{k\in[n]}t_{k, 1}C_k=TD_1 R$. Note that $t_{i,1}\neq 
0$, and $C_1'=C_i$. We have $\colspan(C_1')=\colspan(C_i)\subseteq 
\colspan(\sum_{k\in[n]}t_{k, 1}C_k)=\colspan(TD_1)=\colspan(T'D_1)$, where the inclusion is due to~\cref{lem:conj_key} and the
last equality is due to the assumption that $e_1\not\in \colset(D_1)$.

\item[(d)] Suppose $e_1\in\colset(D_1)$. This is when both changes happen: the 
first 
frontal slice changes from $\sum_{k\in[n]}t_{k, 1}C_k$ to $C_i$, and 
$t_1\in\colset(TD_1)$ changes to $e_i\in\colset(T'D_1)$. Still, we have $\colspan(\sum_{k\in[n]}t_{k, 1}C_k)= \colspan(T'D_1)$ by the argument for case (b), and $\colspan(C_1')=\colspan(C_i)\subseteq\colspan(\sum_{k\in[n]}t_{k, 1}C_k)$ by the argument for case (c). These allow us to deduce that 
$\colspan(C_1')\subseteq \colspan(T'D_1)$.\qedhere
\end{enumerate}
\end{proof}
The proof of~\cref{claim:conj_technical} concludes the proof of~\cref{prop:conjugate}.
\end{proof}

The proof of~\cref{prop:conjugate} also yields the following.
\begin{corollary}\label{lem:conj_conseq}
Let $G=([n], E)$ and $H=([n], F)$ be directed graphs. Let $\cS_G$ and 
$\cS_H$ be 
the corresponding 
graphical matrix spaces, respectively. Let $T=(t_{i,j})\in\GL(n, \F)$ 
satisfy that $T\cS_GT^{-1}=\cS_H$. For $i, j\in[n]$, let $T_{(i,j)}$ be the 
submatrix of $T$ obtained by deleting the $i$th row and the $j$th column. If 
$t_{i,j}\neq0$ and $T_{(i,j)}$ is invertible, then there exists an isomorphism $\pi$ from $G$ to $H$, 
such that $\pi(i)=j$. 
\end{corollary}
In contrast to~\cref{prop:congruence}, there is no embedding version of this correspondence: the following example illustrates that there exist directed graphs $G$ and $H$ such that $G$ is not isomorphic to any subgraph of $H$ while $\cS_G$ is conjugate to a subspace of $\cS_H$.
\begin{example}\label{example: no embedding version of conjugacy}
Let $G=([3],\{(1,1),(2,2)\})$ and $H=([3],[3]\times[3]\setminus\{(1,1),(3,3)\})$. Their graphical matrix spaces are 
\[
\cS_G=\left\{\begin{bmatrix}a&0&0\\0&b&0\\0&0&0\end{bmatrix}:~a,b\in\F\right\}~\text{and}~\cS_H=\left\{\begin{bmatrix}0&x_1&x_2\\y_1&z&x_3\\y_2&y_3&0\end{bmatrix}:~x_1,x_2,x_3,y_1,y_2,y_3,z\in\F\right\}.
\]
As $G$ has two self-loops while $H$ only has one, $G$ is not isomorphic to any subgraph of $H$. However, setting $T=\begin{bmatrix}1&0&1\\1&1&1\\0&1&1\end{bmatrix}\in\GL(3,\F)$, we have 
\[
T\cS_GT^{-1}=\left\{\begin{bmatrix}0&a&-a\\-b&a+b&-a\\-b&b&0\end{bmatrix}:~a,b\in\F\right\},
\]
which is a subspace of $\cS_H$.
\end{example}

\section{Vertex transitivity and conjugacy/congruence irreducibility}\label{subsec:transitivity}

Let $G=([n], E)$ be a directed graph. Let $\Aut(G)\leq\S_n$ be the automorphism 
group of $G$. Recall that $G$ is \emph{vertex-transitive}, if $\Aut(G)$ is a 
transitive 
group. 
Let $\cS\leq\M(n, \F)$. Let $\Conj(\cS):=\{T\in\GL(n, \F) \mid T\cS 
T^{-1}=\cS\}\leq\GL(n, \F)$. We say that $\cS$ is \emph{conjugacy irreducible}, if 
$\Conj(\cS)$ is irreducible as a matrix group. Let 
$\operatorname{Cong}(\mathcal{S}):=\{T\in\operatorname{GL}(n, \mathbb{F}) \mid 
T\mathcal{S} T^{t}=\mathcal{S}\}\leq\operatorname{GL}(n, \mathbb{F})$. We say that 
$\mathcal{S}$ is \emph{congruence irreducible}, if 
$\operatorname{Cong}(\mathcal{S})$ is irreducible as a matrix group.

\begin{proposition}\label{prop:transitive}
	Let $\F$ be a field of order $>2$. Let $G=([n], E)$ be a directed 
	graph, and $\cS_G\leq\M(n, \F)$ be the matrix space 
	associated with $G$. If $G$ is vertex-transitive, then $\cS_G$ is 
    conjugacy irreducible and congruence irreducible.
\end{proposition}
\begin{proof}
	Let $P=\Aut(G)\leq\S_n$ be a transitive group. Embed 
	$P$ as a subgroup of $\GL(n, \F)$, and let $D$ be the group of invertible diagonal matrices 
	in $\GL(n, \F)$. By $P, D\leq\Conj(\cS_{G})$ (resp.\ $P, D\leq\Cong(\cS_{G})$), the group $H=\langle P, D\rangle$ is in 
	$\Conj(\cS_{G})$ (resp.\ $\Cong(\cS_{G})$). Note that $H$ is a subgroup of the monomial group. We claim that $H$ 
	is irreducible. To see this, by the action of $D$ and $|\F|>2$,\footnote{If $|\F|=2$, $H$ is reducible because of the all-one vector.} the only possible 
	invariant subspaces are sums of coordinate subspaces. Then the action of $P$ 
	ensures that these invariant subspaces cannot be proper. This proves the claim and implies $\cS_{G}$ is conjugacy irreducible and congruence irreducible.
\end{proof}

Now we prove that conjugacy and congruence irreducibility imply vertex-transitivity. To this end, we first show the graphs are regular and relate the automorphisms of graphs to the invertible matrices in $\Conj(\cS_{G})$ and $\Cong(\cS_{G})$. First of all, we recall some notation. Let $\tS_G\in\T(n\times n\times m, \F)$, 
$\tC\in\T(n\times m\times n, \F)$, and $C_i\in \M(n\times m, \F)$ be as in~\cref{def:conj}. Note that $C_i$ records the information of the 
out-neighbors of the vertex $i$, and $\rank(C_i)=\dim(\colspan(C_i))$ is equal to the out-degree of $i$.
Let $T\in\Conj(\cS_G)$ (resp.\ $T\in\Cong(\cS_G)$) and write $T=\begin{bmatrix}
t_1 \\ \vdots \\ t_n
\end{bmatrix}$, where $t_i=(t_{i,1}, \dots, t_{i,n})$ is the $i$th row of $T$. Then $\tC^{T^t}=T^t\tC R$ (resp.\ $\tC^{T^t}=T^{-1}\tC R$)
for some $R\in\GL(m, \F)$. Denote by $C_i'$ the $i$th frontal slice of $\tC^{T^t}$.

\begin{lemma}\label{lem:regularity}
	Let $G=([n], E)$ be a directed graph, and $\cS_G\leq\M(n, \F)$ be the matrix space associated with $G$. If $\cS_G$ is conjugacy irreducible or congruence irreducible, then $G$ is regular.
\end{lemma}
\begin{proof}

We first show that $\cS_G$ being conjugacy irreducible (resp.\ congruence 
irreducible) implies that $G$ is out-regular, i.e., every vertex is of the same 
out-degree. By way of contradiction, let $S\subseteq [n]$ be the set of vertices 
of the smallest out-degree $d$. Then for any given $i\in S$, 
$\dim(\colspan(T^tC_i))=\rank(C_i)=d$ (resp.\ 
$\dim(\colspan(T^{-1}C_i))=\rank(C_i)=d$). By~\cref{lem:conj_key} 
(resp. \cref{lem:conj_key} with changing the conclusion to 
$\colspan(C_i')\subseteq \colspan(T^{-t}D_i)$), $\dim(\colspan(C_i'))\leq d$. 

We claim that the 
indices of the non-zero entries of $t_i$ must be in $S$, i.e., $t_{i,j}=0$ for any 
$j\notin S$. To see this, recall that $C_i'=\sum_{k=1}^{n}t_{i,k}C_k$. So if for 
some $j\notin 
S$, $t_{i, j}\neq 0$, then this $C_j$ with $\rank(C_j)>d$ would be involved in 
$C_i'$, which 
would result in $\dim(\colspan(C_i'))>d$ by~\cref{obs:non-overlapping}. 

Thus, we have 
$t_{i,j}=0$ for any $i\in S$ and $j\notin S$. It follows that $\langle e_{i}, i\in 
S\rangle$ is an invariant subspace of $T$, giving the desired contradiction.

Recall that $C_i$'s are the horizontal slices of $\tS_G$. Applying the same 
argument to the vertical slices of $\tS_G$, we can deduce that $G$ is also 
in-regular, i.e., every vertex is of the same in-degree. It follows $G$ is regular.
\end{proof}

\begin{lemma}\label{same_neighbor}
	Let $G=([n], E)$ be a directed graph and $\cS_G\leq\M(n, \F)$ be the matrix space associated with $G$. If $\cS_G$ is conjugacy irreducible (resp.\ congruence irreducible), then for $T\in\Conj(\cS_G)$ (resp.\ $T\in\Cong(\cS_G)$) and $i,j\in[n]$, we have
	\begin{itemize}
		\item[(1)] vertices $k$ and $\ell$ share the same out-neighborhood if 
		$t_{i,k}$ and $t_{i,\ell}$ are non-zero;
		\item[(2)] vertices $k$ and $\ell$ share the same in-neighborhood if 
		$t_{k, j}$ and $t_{\ell, j}$ are non-zero.
	\end{itemize}
\end{lemma}
\begin{proof}
We first prove $(1)$ by contradiction. By~\cref{lem:regularity}, $G$ is 
$d$-out-regular, which implies that $\dim(\colspan(T^tC_i))=\rank(C_i)=d$ (resp.\ 
$\dim(\colspan(T^{-1}C_i))=\rank(C_i)=d$) for any $i\in[n]$. Now assume that 
vertex $k$ and $\ell$ do not share the same out-neighborhood. 
Thus, $\colspan(C_{k})\neq\colspan(C_{\ell})$. 
Note that $C_i'=\sum_{j=1}^{n}t_{i,j}C_j$ where $t_{i,k}$ and $t_{i,\ell}$ 
are non-zero. (Recall that $C_i'$ the $i$th frontal slice of $\tC^{T^t}$.) 
By~\cref{obs:non-overlapping}, we have
$\dim(\colspan(C_i'))>d$. This leads to a contradiction to the statement of 
\cref{lem:conj_key} (resp. \cref{lem:conj_key} with changing the conclusion 
to $\colspan(C_i')\subseteq \colspan(T^{-t}D_i)$). The proof for $(2)$ is the same 
by considering the vertical slices.
\end{proof}

We are now ready to prove that conjugacy irreducibility or congruence 
irreducibility implies vertex-transitivity.

\begin{proposition}
	Let $\F$ be a field of order $>2$. Let $G=([n], E)$ be a directed 
	graph, and $\cS_G\leq\M(n, \F)$ be the matrix space 
	associated with $G$. If $\cS_G$ is conjugacy irreducible, $G$ is vertex 
	transitive.
\end{proposition}
\begin{proof}
    Recall that $T=(t_{i,j})\in\Conj(\cS_G)$. We use $T_{(i,j)}$ to denote the 
    matrix by deleting the 
	$i$th row and $j$th column from $T$, and $T_{(i,j),(i^{\prime},j^{\prime})}$ 
	to denote the matrix by deleting the $i$th and $i^{\prime}$th rows as well as 
	$j$th and $j^{\prime}$th columns from $T$, and so on. 
	
	By way of contradiction, suppose $G$ is not vertex-transitive. Then the action 
	of $\aut(G)\leq\S_n$ on $[n]$ has at least two orbits.
    Let $S\subset [n]$ be one orbit of $\Aut(G)$. 
    By a permutation of the vertices if necessary, we can set 
	$S=[s]$ for $s=|S|$. 
	\begin{claim}\label{claim:zero_block}
		For any $i\in[s]$ and $j\in[n]\setminus [s]$, $t_{i,j}=0$.
	\end{claim}
	\begin{proof}
		By contradiction and without loss of generality, assume $t_{1, j}\neq 0$ for some 
		$j\geq s+1$. Note that no automorphism of $G$ can send any $i\in [s]$ to any 
		$j\in [n]\setminus [s]$ and vice versa. By~\cref{lem:conj_conseq}, $T_{(1, 
			j)}$ must be singular. On the other hand, since $T$ is invertible, there exists $i\in[n]$ such that $t_{1, i}\neq 0$ and $T_{(1,i)}$ is invertible. 
		Then we can find $\ell\in[n]\setminus\{1\}$ such that $t_{\ell,j}\neq0$ and $T_{(1, i), (\ell, j)}$ is invertible. 
		Let $T(a)$ be the matrix obtained by replacing the 
		$(\ell, i)$th entry in $T$ with $a\in \F$. By the Laplace expansion of $T(a)_{(1,j)}$ with respect to the $i$th column, it follows that $\det(T(a)_{(1, 
			j)})=\lambda a+\gamma$ for some $\gamma\in \F$, where 
		$\lambda=\det(T_{(1,j),(\ell,i)})=\det(T_{(1,i),(\ell,j)})\neq 0$. This means that $T(a)_{(1, j)}$ is 
		invertible for all but one of $a\in\F$. Similarly, $\det(T(a))=\lambda' 
		a+\gamma'$ for some $\lambda'=\det(T_{(\ell,i)})$ and $\alpha'\in\F$, which is not the zero polynomial. So for all but at most two of $a\in \F$, 
		$\det(T(a))$ and $\det(T(a)_{(1, j)})$ are both nonzero. 
        Since $|\F|>2$, we can fix some $a\in \F$ such that $T(a)$ and $T(a)_{(1, 
		j)}$ are both invertible. Note that the 
		$(1,j)$th entry of $T(a)$ is $t_{1,j}\neq 0$. If $T(a)\in \Conj(\cS_G)$, by~\cref{lem:conj_conseq}, 
		we can find an automorphism $\pi$ of $G$ sending $1$ to $j$, which would
		contradict the assumption of orbits of $\Aut(G)$.
		Recall that $C_i'$ is the $i$th frontal slice of $\tC^{T^t}$ and denote by $C_i^{\prime}(a)$ the $i$th frontal slice of $\tC^{T(a)^{t}}$. We are going to use~\cref{lem:conj_key} to show $T(a)\in \Conj(\cS_G)$. To this end, we shall prove that, changing $T$ to $T(a)$ does not affect $\colspan(C_{k}^{\prime})$ and $\colspan(TC_{k})$ for any $k\in[n]$.
		
		Firstly, we claim that 
		$\colspan(C_{k}^{\prime})=\colspan(C_{k}^{\prime}(a))$ for any $k\in[n]$. 
		Since the only difference between $T$ and $T(a)$ is the $(\ell, i)$th entry, 
		the claim clearly holds for all $k\in[n]\setminus\{\ell\}$. Furthermore, the only 
		difference between $C_{\ell}^{\prime}$ and $C_{\ell}^{\prime}(a)$ as the 
		linear combination of $C_{k}$'s is the coefficient of $C_{i}$, and the only problematic choice is when $a=0$. To see this is not a problem,
		note that $t_{1,i}$ and $t_{1,j}$ are both non-zero, by (1) of~\cref{same_neighbor}, vertices 
		$i$ and $j$ share the same out-neighborhood in $G$, thus 
		$\colspan(C_{i})=\colspan(C_{j})$. 
		Since $t_{\ell,j}\neq0$, $C_{j}$ is involved in both the linear combinations of $C_{\ell}^{\prime}$ and $C_{\ell}^{\prime}(a)$, which 
		implies that 
		$\colspan(C_{\ell}^{\prime})=\colspan(C_{\ell}^{\prime}(a))$. 
		
		Secondly, we 
		claim that $\colspan(T(a)C_{k})=\colspan(TC_{k})$ for any $k\in[n]$. To see 
		this, note that $T(a)C_k$ and $TC_k$ are different if and only if $\ell$ is in 
		the out-neighborhood of $k$; in other words, $k$ is in the in-neighborhood 
		of $\ell$. 
		Suppose this holds, and the only different column is $t_\ell^t$ and $t_\ell^t(a)$ (whose $i$th coordinate is replaced by $a$). 
		Since $t_{1, j}$ and $t_{\ell, j}$ are both 
		non-zero, ~\cref{same_neighbor} (2) ensures that the in-neighborhoods of 
		$1$ and $\ell$ are the same. It follows that $1$ is also in the 
		out-neighborhood of $k$, which means that $t_1^t$ and $t_\ell^t$ are both 
		in 
		$TC_k$. Recall that the columns of $C_k$ are standard column basis vectors and $C_k'=\sum_{k'\in[n]}t_{k,k'}C_{k'}$. By~\cref{obs:non-overlapping}, $\colspan(C_k')$ is spanned by the standard basis (or more precisely, the set $\{e^t_{k'}:t_{k,k'}\neq 0~\&~(k,k')\in E\}$). Since $\colspan(TC_k)=\colspan(C_k')$ and the $i$th coordinate of $t_1$ is non-zero, $e^t_i\in 
		\colspan(C_k')$. So changing the $i$th coordinate of $t_\ell$ to $a$, as long as the 
		dimension of $\colspan(T(a)C_{k})$ does not decrease, does not affect $\colspan(TC_k)$. These conclude that
		$\colspan(TC_k)=\colspan(T(a)C_k)$. 
		
		Since $T(a)\in \Conj(\cS_G)$, $T(a)_{(1, j)}$ are invertible, and the $(i,j)$th entry of $T(a)$ is $t_{i,j}\neq0$, by~\cref{lem:conj_conseq}, there exists an automorphism $\pi$ of $G$ sending $1$ to $j$, which gives us the desired contradiction.
    \end{proof}
	\cref{claim:zero_block} implies that all $T\in\Conj(\cS_G)$ are reducible with a common invariant subspace $\langle e_{i}, i\in S\rangle$, and thereby contradicts $\cS_G$ being conjugacy irreducible.
\end{proof}

\begin{proposition}
	Let $G=([n], E)$ be a directed 
	graph, and $\cS_G\leq\M(n, \F)$ be the matrix space 
	associated with $G$. If $\cS_G$ is congruence irreducible, $G$ is vertex 
	transitive.
\end{proposition}
\begin{proof}
    We first claim that $G$ must either have all-self-loops or no self-loops. If not, up to a permutation of vertices, we can assume that $\E_{i,i}\notin\mathcal{S}_{G}$ but $\E_{j,j}\in\mathcal{S}_{G}$ for all $i\in[s]$ and $j\in[n]\setminus[s]$. Note that for $T\in\Cong(\mathcal{S}_{G})$ and $j\in[n]\setminus[s]$, $T\E_{j,j}T^{t}=[t_{k,j}t_{\ell,j}]_{k,\ell\in[n]}\in\cS_G$. Since $\E_{i,i}\notin\mathcal{S}_{G}$ for $i\in[s]$, $t_{i,j}t_{i,j}=0$ for any $(i,j)\in [s]\times([n]\setminus[s])$.
    This implies that $\Cong(\cS_G)$ is reducible with a common invariant subspace $\langle e_{i}\mid i\in [s]\rangle$, which is a contradiction. 

    Now we claim that, for $T\in\Cong(\mathcal{S}_{G})$, if $t_{i,j}\neq0$ for 
    some $i\neq j$, then there exists an automorphism $\pi$ of $G$, such that 
    $\pi(i)=j$. Since $T\in\Cong(\mathcal{S}_{G})$, by~\cref{cor:cong_conseq}, 
    there exists an automorphism $\sigma$ of $G$ such that $t_{i,\sigma(i)}\neq0$ 
    and $t_{\sigma^{-1}(j),j}\neq0$. If $\sigma(i)= j$, then we are done. So 
    assuming $\sigma(i)\neq j$ and $\sigma^{-1}(j)\neq i$, we embed 
    $\sigma^{-1}\in\Aut(G)$ as a matrix $P\in\GL(n,\F)$. Then we have 
    $PT\in\Cong(\mathcal{S}_{G})$ with the $(\sigma(i),\sigma(i))$th entry, the 
    $(j,j)$th entry, and the $(\sigma(i),j)$th entry all being non-zero. 
    By~\cref{same_neighbor}, it implies that vertex $\sigma(i)$ and $j$ share the 
    same (in- and out-)neighborhood. Since $G$ either has all-self-loops or no 
    self-loops, the transposition of $\sigma(i)$ and $j$, denoted as $\sigma'$, is 
    an automorphism of 
    $G$. It follows that $\pi=\sigma^{\prime}\circ\sigma$ is an automorphism of 
    $G$ sending $i$ to $j$.

	We prove the proposition by contradiction. Suppose $\cS_G$ is congruence 
	irreducible but $G$ is not vertex-transitive. Then $\aut(G)$ has at least two 
	orbits. Let $S\subset [n]$ be one orbit of $\aut(G)$. 
    Then for any $T\in\Cong(\mathcal{S}_{G})$ and for any $(i,j)\in 
	S\times([n]\setminus S)$, the $(i,j)$th entry of $T$ must be zero, 
    as otherwise there exists an automoprhism of $G$ sending $i$ to $j$ by the 
    claim in the last paragraph, contradicting $i$ and $j$ are in different orbits 
    of $\Aut(G)$. 
    It follows that 
	$\Cong(\cS_G)$ is reducible due to the common invariant subspace $\langle 
	e_{i}\mid i\in S\rangle$, contradicting the assumption that $\cS_G$ is 
	congruence irreducible.
\end{proof}

\section{Connections to quantum information theory}\label{sec: quantum}

In this section, we demonstrate connections of our results with some known 
results on quantum channels. We put such connections into two categories, 
which correspond to the two main approaches of viewing quantum channels as  
generalizations of graphs. 

Before going into the details, we collect some basic notions from quantum 
information theory. To be consistent with the quantum information literature, we 
identify $u\in\mathbb{C}^n$ as \emph{column vectors} and matrices $A\in 
M(n,\mathbb{C})$ acting on $u$ \emph{from the left}, i.e.,~$Au\in\mathbb{C}^n$. 
Recall that a linear map $\Phi: \M(n,\mathbb{C})\to \M(n',\mathbb{C})$ is 
\emph{completely positive} (CP) if for any $d\in\N$ and any positive semidefinite 
matrix 
$X\in\M(d,\mathbb{C})\otimes 
\M(n,\mathbb{C})$, the matrix 
$(\mathrm{id}_d\otimes\Phi)(X)\in\M(d,\mathbb{C})\otimes\M(n',\mathbb{C})$ is 
positive semidefinite, where $\mathrm{id}_d:\M(d,\mathbb{C})\to\M(d,\mathbb{C})$ 
is the identity map on $\M(d, \mathbb{C})$. 
Every CP linear map $\Phi: \M(n,\mathbb{C})\to \M(n',\mathbb{C})$ admits the 
Choi--Kraus representation as follows. There exist matrices 
(Choi--Kraus operators) $E_1,\dots,E_m\in \M(n'\times n,\mathbb{C})$ such that 
$\Phi(X)=\sum_{i=1}^m E_i X E_i^*$. While Choi--Kraus representations of a CP map 
$\Phi$ may not be unique, they span the same matrix space 
$\cS_\Phi=\langle E_1,\dots, E_m\rangle$. We say $\Phi$ is \emph{trace-preserving} 
(TP) if 
$\operatorname{Tr}(\Phi(X))=\operatorname{Tr}(X)$ for any $X\in \M(n,\mathbb{C})$, and $\Phi$ is 
\emph{unital} if $\Phi(I_n)=I_{n'}$. In terms of Choi--Kraus operators, $\Phi$ is 
trace-preserving if $\sum_{i=1}^m E_i^*E_i=I_n$ and $\Phi$ is unital if 
$\sum_{i=1}^m E_iE_i^*=I_{n'}$. Quantum channels are trace-preserving CP maps, and unital quantum channels are natural quantum generalizations of doubly stochastic matrices. 

\subsection{Quantum channels and transition matrices} 

A transition matrix $P=(p_{i,j})_{i,j\in[n]}\in \M(n,\R_{\geq 0})$ is a column 
stochastic matrix, i.e.,~it satisfies that $\sum_{i=1}^n p_{i,j}=1$ for any 
$j\in[n]$. Such a matrix $P$ describes a discrete Markov chain where the 
probability of moving from $j$ to $i$ is given by $p_{i,j}$.\footnote{In the 
theory of discrete Markov chains, the transition matrix normally acts on 
probability row vectors from right. For consistency with quantum channels, we 
translate the action to acting from the left on probability column vectors; see 
also~\cite{quasirandom}.} The underlying directed graph $G$ of a transition matrix 
$P$ is naturally defined with vertex set $[n]$ and arc set $E=\{(i,j) \mid 
p_{j,i}\neq 0\}$.

Quantum channels can be viewed as a natural quantum generalization of transition 
matrices. Indeed, associate each transition matrix $P=(p_{i,j})_{i,j\in[n]}$ with 
the CP map $\Phi_P: \M(n,\mathbb{C})\to\M (n,\mathbb{C})$ defined as 
\begin{equation}\label{eq: CMC to QMC}
\Phi_P(X)=\sum_{i,j=1}^n p_{i,j}\E_{i,j} X \E_{i,j}^*
\end{equation}
for any $X\in \M(n,\mathbb{C})$. It is clear that 
\[
\Phi_P(I_n) = \sum_{i,j=1}^n p_{i,j} \E_{i,j}^*\E_{i,j}=\sum_{i,j=1}^n p_{i,j} \E_{j,j}=\sum_{j=1}^n \left(\sum_{i=1}^n p_{i,j}\right) \E_{j,j}=\sum_{j=1}^n \E_{j,j}= I_n,
\]
where the second last equality uses the fact that $P$ is column stochastic. Thus $\Phi_P$ is a quantum channel. To realize the action of $P$ on probability vectors, we apply $\Phi_P$ on diagonal density matrices (positive semidefinite matrices with unit trace)~\cite{quasirandom}. Observe that $\cS_{\Phi_P}=\langle \sqrt{p_{i,j}} \E_{i,j}\mid p_{i,j}\neq 0\rangle=\cS_G^*$,
where $G$ is the underlying directed graph of $P$ and the adjoint space $\cS^*$ of 
$\cS$ is defined as $\{B^*:~B\in\cS\}$.
This illustrates a natural way to generalize directed graphs to 
quantum channels. 

\paragraph{Irreducibility of transition matrices.}
We first connect the irreducibility of a quantum channel $\Phi:\M(n, \mathbb{C})\to\M(n, 
\mathbb{C})$ with the 
irreducibility of its underlying matrix space $\cS_\Phi$. A quantum channel $\Phi$ 
is 
\emph{irreducible}, if the only orthogonal projections $P$ (i.e., $P^2=P$ and $P^*=P$) 
satisfying $\Phi(P\M(n,\mathbb{C})P)\leq P\M(n,\mathbb{C})P$ are $0$ and 
$I_n$ (cf.~\cite{EH78,QCO}). We first observe that the irreducibility notion for 
quantum channels 
coincides with the irreducibility notion for their underlying matrix spaces. 
\begin{proposition}\label{obs: channel irred. and space irred.}
A CP map $\Phi:\M(n, \mathbb{C})\to\M(n, \mathbb{C})$ is irreducible if and only if $\cS_\Phi$ is 
irreducible.
\end{proposition}
\begin{proof}
Let $E_1, \dots, E_m$ be a set of Choi--Kraus operators for $\Phi$. 
We first show that $\Phi$ being reducible implies that $\cS_\Phi$ is reducible. 
Let $P$ be a nontrivial projection such that 
$\Phi(P\M(n,\mathbb{C})P)\leq P\M(n,\mathbb{C})P$. Let $u_1,\dots,u_k\in \mathbb{C}^n$ be 
an 
orthonormal basis of $\colspan(P)$, so $P\M(n,\mathbb{C})P=P\M(n, \mathbb{C})P^*=\langle 
u_iu_j^*\mid i, j\in[k]\rangle$. Then for each $j\in[k]$,
\[
\Phi(u_ju_j^*)=\sum_{i=1}^m E_iu_ju_j^*E_i^*\in \langle u_iu_j^*\mid 
i,j\in[k]\rangle.
\]

We now claim that $\colspan(\Phi(u_ju_j^*))= \langle E_iu_j\mid i\in[m]\rangle$.
First, it is straightforward to verify that $\colspan(\Phi(u_ju_j^*))\leq \langle 
E_iu_j\mid i\in[m]\rangle$. Second, for any $v\in\colspan(\Phi(u_ju_j^*))^\perp$, 
we have $v^*\Phi(u_ju_j^*)v=0$. Expanding the left-hand side, we have
\[
0=v^*\Phi(u_ju_j^*)v=\sum_{i=1}^m v^*E_iu_ju_j^*E_i^*v=\sum_{i=1}^m |v^*E_iu_j|^2.
\]
Thus $v\in\langle E_iu_j\mid i\in[m]\rangle^\perp$, showing that 
$\colspan(\Phi(u_ju_j^*))^\perp\leq \langle E_iu_j\mid i\in[m]\rangle^\perp$. The 
claim is then proved. 

It follows that $\colspan(\Phi(u_ju_j^*))=\langle E_iu_j\mid 
i\in[m]\rangle\leq\colspan(P)$ 
for every $j\in[k]$. This implies that $\colspan(P)$ is a nontrivial invariant 
subspace of $\cS_\Phi$, thus $\cS_\Phi$ is reducible.

We now show that $\cS_\Phi$ being reducible implies that $\Phi$ is reducible. Let 
$U$ be a nontrivial invariant subspace of $\cS_\Phi$ of dimension $d$ 
and let $P$ be the orthogonal projection on $U$. It is not hard to verify that $\Phi(P\M(n, \mathbb{C})P)\leq P\M(n, \mathbb{C})P$.
\end{proof}

Recall that a transition matrix $P$ is irreducible if its 
underlying directed graph $G$ is strongly connected. Thus,~\cref{thm:conn} 
formally bridges the irreducibility of transition matrices and quantum channels.
\begin{corollary}\label{cor: equivalence quantum classical irreducibility}
Let $P$ be a transition matrix, $G$ be its underlying directed graph and $\Phi_P$ be the associated quantum channel given in~\cref{eq: CMC to QMC}. Then $P$ is irreducible if and only if $\Phi_P$ is irreducible.
\end{corollary}
\begin{proof}
Let $G^t=([n],E^t)$ be the transpose graph of $G=([n],E)$, where~$E^t=\{(i,j)\mid~(j,i)\in E\}$. Then $\cS_G^*=\cS_{G^t}$. Moreover, $G$ is strongly connected if and only if $G^t$ is strongly connected. The equivalence then follows from~\cref{thm:conn} and $\cS_{\Phi_P}=\cS_G^*=\cS_{G^t}$.
\end{proof}

\paragraph{Vertex transitivity.}
We then turn to vertex transitivity. Since every vertex-transitive graph $G$ is 
regular, its normalized adjacency matrix $A$ is a symmetric transition matrix. We 
can associate to $G$ the following quantum channel $\Phi_G$ in the spirit 
of~\cref{eq: CMC to QMC}:
\begin{equation}\label{eq: qc of regular graph}
\Phi_G(X)=\frac{1}{d}\sum_{(i,j)\in E} \E_{i,j}X\E_{i,j}^*
\end{equation}
for every $X\in \M(n,\mathbb{C})$, where $d$ denotes the in-degree of any vertex in $G$. It is clear that $\cS_{\Phi_G}=\cS_G$.

An interesting result in~\cite{quasirandom} connects vertex-transitive graphs with 
irreducibly covariant channels. A 
quantum channel $\Phi: \M(n,\mathbb{C})\to \M(n,\mathbb{C})$ is \emph{irreducibly 
covariant}, if there exists a compact group $\Gamma$ and a continuous irreducible 
unitary representation $U:\Gamma\to \mathrm{U}(n)$ such that for any $g\in\Gamma$ 
and $X\in \M(n,\mathbb{C})$, we have $\Phi(U(g)XU(g)^*)=U(g)\Phi(X)U(g)^*$. 
\begin{proposition}[\hspace{1sp}{\cite[Proposition 
3.8]{quasirandom}}]\label{prop:quasirandom}
$G$ is vertex-transitive if and only if $\Phi_G$ (as defined in~\cref{eq: qc of 
regular graph}) is irreducibly covariant.
\end{proposition}
We can deduce~\cref{prop:quasirandom} easily from \cref{thm:trans} 
and~\cref{prop:transitive}. 
\begin{proof}
For the if direction, let $U:\Gamma\to \mathrm{U}(n)$ be the irreducible unitary 
representation of some compact group $\Gamma$ such that for any $g\in\Gamma$ and 
$X\in \M(n,\mathbb{C})$, we have $\Phi(U(g)XU(g)^*)=U(g)\Phi(X)U(g)^*$. In this 
case, for every $g\in \Gamma$, the map $\Phi': \M(n,\mathbb{C}) \to 
\M(n,\mathbb{C})$ given by $\Phi'(X)=\Phi(U(g)XU(g)^*)$ and the map $\Phi'': 
\M(n,\mathbb{C}) \to \M(n,\mathbb{C})$ given by $\Phi''(X)=U(g)\Phi(X)U(g)^*$ are 
equivalent. This implies that 
$\cS_{\Phi_G}U(g)=\cS_{\Phi'}=\cS_{\Phi''}=U(g)\cS_{\Phi_G}$. It follows that 
$U(\Gamma)\leq\Conj(\cS_G)$, and, since $U(\Gamma)$ is irreducible, $\Conj(\cS_G)$ 
is irreducible. 
By~\cref{thm:trans}, we can conclude that $G$ is vertex-transitive.

For the only if direction, since $G$ is vertex-transitive,~\cref{prop:transitive} proves that the matrix group $H$ generated by the automorphism group of $G$ (embedded as a subgroup of $\GL(n,\mathbb{C})$) and the group of diagonal unitary matrices is irreducible. It is straightforward to verify that $H$ is compact and $\Phi(UXU^*)=U\Phi(X)U^*$ for any $U\in H$.
\end{proof}

Summarizing \cref{thm:trans} and \cref{prop:quasirandom}, we have the following.
\begin{corollary}\label{cor: vertex-transitive}
The following are equivalent:
\begin{itemize}
\item[(1)] $G$ is vertex-transitive;
\item[(2)] $\Phi_G$ is irreducibly covariant;
\item[(3)] $\cS_G$ is conjugacy irreducible;
\item[(4)] $\cS_G$ is congruence irreducible.
\end{itemize}
\end{corollary}

\paragraph{Discussion: Quantum (spectral) expanders and dimension expanders.} 
Expander graphs are an important family of graphs which are sparse while highly 
connected. 
A natural expansion property for matrix spaces is the spectral expansion, which is mostly studied for \emph{quantum expanders}~\cite{PhysRevA.76.032315,DBLP:journals/qic/Harrow08,BST10,quantumexpander}. Recall that the spectral expansion of a $d$-regular graph $G$ is the second largest absolute value of the eigenvalues of its adjacency matrix $A_G$. On the other hand, from a $d$-regular graph $G=([n],E)$, we have the linear map $\Phi_G: \M(n,\mathbb{C})\to \M(n,\mathbb{C})$ (as given in~\cref{eq: qc of regular graph}), such that $\cS_G$ is the matrix space spanned by the Choi--Kraus operator of $\Phi_G$. Define the spectral expansion of a linear map $\Phi:\M(n,\mathbb{C})\to\M(n,\mathbb{C})$ as the second largest absolute value of its eigenvalues. 
\begin{theorem}[\hspace{1sp}{\cite[Proposition 3.7]{quasirandom}}]\label{thm: spectral expander}
For any $d$-regular graph $G$, the spectral expansion of $G$ equals the spectral expansion of $\Phi_G$ (as given in~\cref{eq: qc of regular graph}).
\end{theorem}

Another type of expansion property for matrix spaces is dimension 
expansion. We say a matrix space $\cS\leq \M(n,\F)$ is a degree-$d$ 
\emph{dimension 
expander} if $\dim(\cS)=d$ and there exists a constant $\alpha$ such that for any 
subspace $V\leq \F^n$ with $\dim(V)\leq n/2$, we have 
$\dim(\cS(V))\geq(1+\alpha)\dim(V)$. 

Dimension expanders have been shown to connect with \emph{monotone 
expanders}~\cite{10.1007/s00493-011-2540-8,v006a012}. Let $\mathcal{F}$ be a set 
of $d$ (partial) monotone functions $f_1,\dots,f_d:[n]\to[n]$. Let 
$E=\{(i,f_j(i))\mid i\in[n],j\in[d]\}$. We call $G_\mathcal{F}=([n]\times [n],E)$ 
a \emph{degree-$d$ monotone graph} with respect to $\mathcal{F}$. A degree-$d$ 
monotone graph is a degree-$d$ monotone expander if 
there exists a constant $\alpha$ such that every set $A\subseteq [n]$ of left 
vertices of size at most $n/2$ has at least $(1 + \alpha)|A|$ neighbors. 
Given a degree-$d$ monotone graph $G_\mathcal{F}=([n],E)$, construct a matrix 
space $\cS_\mathcal{F}=\langle\sum_{i\in[n]}e_ie_{f_j(i)}^t\mid j\in[d]\rangle$. 

\begin{theorem}[\hspace{1sp}\cite{10.1007/s00493-011-2540-8,v006a012}]\label{thm: dimension expander and monotone expander}
For a set of $d$ monotone functions $\mathcal{F}=\{f_1,\dots,f_d\}$, $G_\mathcal{F}$ is a degree-$d$ monotone expander with constant $\alpha$ if and only if the matrix space $\cS_\mathcal{F}$ is a degree-$d$ dimension expander with constant $\alpha$.
\end{theorem}
Note that the matrix space $\cS_\mathcal{F}$ is different from the graphical matrix space $\cS_{G_\mathcal{F}}$.
Combined with Bourgain's explicit construction of constant-degree monotone 
expanders~\cite{BOURGAIN2009357,Bourgain2013}, one directly obtains an explicit 
construction of constant-degree dimension expanders by~\cref{thm: dimension 
expander and monotone expander}. Note that over fields $\F$ of characteristic $0$, 
one may use other expanders to construct dimension 
expanders~\cite{LUBOTZKY2008730}.

In forthcoming work \cite{forthcoming}, we discuss several additional relationships between the different linear-algebraic notions of expansion.

\subsection{Quantum channels and zero-error communication} 
Another way to connect graphs with quantum channels is via a
quantum version of the zero-error capacity problem. In Shannon's seminal 
work~\cite{MR0089131}, an \emph{undirected} graph is associated to each classical 
communication channel, whose the independence number equals the largest number of zero-error messages one can send through the channel. 
Replacing the classical channel by a quantum channel $\Phi$, the role of graphs is taken by the \emph{operator system}
associated with $\Phi$~\cite{duan2013}. More precisely, let $E_1,\dots, E_m\in 
\M(n'\times n,\mathbb{C})$ be a set of Choi--Kraus operators of $\Phi: 
\M(n,\mathbb{C})\to\M(n',\mathbb{C})$. The operator system of $\Phi$ is
$\opsys_\Phi=\langle E_i^*E_j\mid 
i,j\in[m]\rangle\leq \M(n,\mathbb{C})$, which is self-adjoint 
(i.e.,~$X\in \opsys_\Phi$ implies $X^*\in \opsys_\Phi$) and $I_n\in \opsys$ (since 
$\Phi$ is trace-preserving).
On the other hand, for every self-adjoint $\opsys\leq \M(n,\mathbb{C})$ containing 
$I_n$, one 
can find a quantum channel $\Phi$ such that $\opsys_\Phi=\opsys$. The authors 
of \cite{duan2013} viewed operator systems as a noncommutative generalization of graphs, so called 
them noncommutative graphs.

The correspondence between (undirected) graphs and operator systems is the 
following. 
Let $G=([n],E)$ be an undirected graph. We then define its corresponding 
operator system by 
\begin{equation}\label{eq: nc graph}
\opsys_G=\langle \E_{i,j}\mid \{i,j\}\in E~\text{or}~i=j\in[n]\rangle.
\end{equation}
Note that $\opsys_G$ and the graphical matrix space $\cS_G$ are different. We 
identify each undirected graph $G=([n],E)$ as a directed graph 
$\hat{G}=([n],\hat{E})$, where $\hat{E}=\{(i,j)\mid \{i,j\}\in 
E~\text{or}~i=j\in[n]\}$. In other words, we obtain $\hat{G}$ by identifying each 
undirected edge as two directed edges and adding self-loops to each vertex of $G$. 
Through this identification, we have $\opsys_G=\cS_{\hat{G}}$.  

\paragraph{Connectivity of operator systems.}
The connectivity of operator systems was studied in~\cite{CHAVEZDOMINGUEZ202137}. 
For an operator system $\opsys\leq \M(n,\mathbb{C})$, we say $\opsys$ is 
connected, if there 
is no nontrivial orthogonal projection $P\in \M(n,\mathbb{C})$, such that 
$PS(I_n-P)=\{0\}$ (cf.~\cite[Theorem 3.3]{CHAVEZDOMINGUEZ202137}). 
\begin{corollary}\label{cor: conn}
Let $G=([n],E)$ be an undirected graph and $\hat{G}=([n],\hat{E})$ be the directed graph obtained as above. Then the following are equivalent:
\begin{itemize}
\item[(1)] $G$ is connected;
\item[(2)] $\hat{G}$ is strongly connected;
\item[(3)] $\opsys_G$ is connected;
\item[(4)] $\cS_{\hat{G}}$ is irreducible.
\end{itemize}
\end{corollary}
\begin{proof}
The equivalence between $(1)$ and $(2)$ is direct to verify.
The equivalence between $(1)$ and $(3)$ is shown in~\cite[Corollary 
3.4]{CHAVEZDOMINGUEZ202137} and the equivalence between $(2)$ and $(4)$ is due 
to~\cref{thm:conn}. Note that the connectivity of an operator system $\opsys$ 
exactly implies that $\opsys$ is irreducible. Thus $(3)$ and $(4)$ are equivalent.
\end{proof}

\paragraph{Isomorphic operator systems.} The natural notion of equivalence between 
operator systems is the so-called \emph{unital and complete order isomorphism}. 
Namely, two operator systems $\opsys_1$ and $\opsys_2$ are unital, complete order 
isomorphic if there is a unital linear map $\phi:\opsys_1\to \opsys_2$ which is 
one-to-one, onto and completely positive for both $\phi$ and $\phi^{-1}$. For 
undirected graphs $G$ and $H$, the equivalence between $\opsys_G$ and $\opsys_H$ 
is compatible with the isomorphism of $G$ and $H$, and admits a simple form:
\begin{theorem}[\hspace{1sp}{\cite[Theorem 3.3]{ORTIZ2015128}}]\label{thm: unitary isomorphism}
Let $G$ and $H$ be undirected graphs. Then the following are equivalent:
\begin{itemize}
\item[(1)] $G$ and $H$ are isomorphic;
\item[(2)] $\opsys_G$ and $\opsys_H$ are unital and complete order isomorphic;
\item[(3)] There exists a unitary $U$ such that $U\opsys_GU^*=\opsys_H$.
\end{itemize}
\end{theorem}

We note that \cref{prop:congruence} could be 
viewed as a generalization of the equivalence between $(1)$ and $(3)$ 
in~\cref{thm: unitary isomorphism}, as shown in the following result.
\begin{corollary}\label{cor: isomorphism}
Let $G$ and $H$ be undirected graphs. 
Then the following are equivalent:
\begin{itemize}
\item[(1)] $G$ and $H$ are isomorphic;
\item[(2)] There exists an invertible matrix $T$ such that $T\opsys_GT^*=\opsys_H$;
\end{itemize}
\end{corollary}
\begin{proof}
Let $\hat{G}$ and $\hat{H}$ be the corresponding directed graphs obtained as 
above. 
Note that $G$ and $H$ are isomorphic if and only if $\hat{G}$ and 
$\hat{H}$ are isomorphic.
Then the equivalence between $(1)$ and $(2)$ follows from~\cref{prop:congruence} 
and $\opsys_G=\cS_{\hat{G}}$ and $\opsys_H=\cS_{\hat{H}}$. 
\end{proof}
In \cref{thm: unitary isomorphism}, to show the equivalence between (1) and (3), 
the more difficult direction is to show (3)$\Rightarrow$(1). This is immediate 
from our \cref{cor: isomorphism}.

\section*{Acknowledgement}

Y.Q. would like to thank George Glauberman for helping him with~\cref{prop:transitive}. 

\bibliographystyle{alpha}
\bibliography{references}

\appendix

\section{On finding cycles in directed graphs}

\begin{proposition}\label{prop:cyclicity_matching}
There is an NC-reduction from finding cycles in directed graphs to 
finding perfect matchings in bipartite graphs. 
\end{proposition}
\begin{proof}
Let $G=([n], E)$ be a directed graph. 
We construct a bipartite graph $H=(L\times R, F)$ as follows: Let $L$ and $R$ be 
two copies of $[n]$. To 
distinguish them, denote the vertices of $L$ as $i_{L,1},i_{L,2}$ for all 
$i\in[n]$ and the vertices of $R$ as $j_{R,1},j_{R,2}$ for all $j\in[n]$. The edge 
set $F$ is constructed as follows: For every $(i,j)\in E$, set 
$(i_{L,1},j_{R,1})\in F$. For every $i\in[n]$, let 
$(i_{L,1},i_{R,2}),(i_{L,2},i_{R,2}),(i_{L,2},i_{R,1})\in F$. The second set of 
edges is exactly length-$3$ paths from $i_{L,1}$ to $i_{R,1}$, which will be 
denoted by $P_i$ for each $i\in[n]$.

Consider perfect matchings in $H$, which must take edges from each $P_i$. For a 
fixed perfect matching $M$, divide $[n]$ into two disjoint sets $S$ and $T$, 
where $S$ contains those $i$ where the edge $(i_{L,2},i_{R,2})\in M$, and $T$ 
contains those $i$ where the edges $(i_{L,1},i_{R,2}),(i_{L,2},i_{R,1})\in M$. 
Note that if $S\neq\emptyset$, the edges in $M$ which start from $i_{L,1}$ for 
$i\in S$ will end in some $j_{R,1}$ for $j\in S$, and these edges form a cycle or 
disjoint union of cycles in $G$.

The only problem is $S$ can be empty. To resolve this issue, one can remove the 
edges in $P_1$ from $F$ and repeat the above steps, which will yield cycles in $G$ 
containing $1$ (if there is one). Repeating this for all $i$ yields the 
desired reduction.
\end{proof}
\cref{prop:cyclicity_matching} implies that finding a cycle on a directed 
graph can be done in quasi-NC by~\cite{doi:10.1137/16M1097870}. However, it is not 
hard to show that finding a cycle in a directed graph is in NC. 
\begin{proposition}\label{prop:cycle_nc}
Finding a cycle in a directed graph can be solved in NC.
\end{proposition}
\begin{proof}[Proof sketch]
Let $G=([n], E)$ be a directed graph. Let $A_G$ be the adjacency matrix of $G$. 
For each entry of $A_G$ we store a walk of length $1$. The for $k\in\lceil \log 
n\rceil$, compute $A_G^{2^k}$ recursively using $A_G^{2^{k-1}}$. Furthermore, for 
each entry of $A_G^{2^k}$, if it is non-zero, store a walk of length $2^k$ (by 
concatenating two walks). In this way, if $A_G^{2^{\lceil \log n\rceil}}$ is 
non-zero (which happens if and only if $G$ is cyclic), we obtain a walk of length 
$\geq n$, from which we can extract a cycle easily. All the above steps run in 
$\mathrm{NC}^2$.
\end{proof}

\begin{proposition}\label{prop:nil_sing}
Matrix space nil testing reduces to symbolic determinant identity testing.
\end{proposition}
\begin{proof}
Let $\cS=\langle B_1, \dots, B_m\rangle \leq\M(n, \F)$ be a matrix space. Let 
$x_1, \dots, x_m$ be variables, and construct a symbolic matrix $\vB:= 
x_1B_1+\dots+x_mB_m$. Note that $\cS$ is nil if and only if $\vB^n$ is the 
all-zero matrix. By~\cite{FS13}, testing if a particular entry of $\vB^n$ is the 
zero polynomial reduces to SDIT. So we can use $n^2$ instances of SDIT to solve 
the matrix space nil testing problem.
\end{proof}

\end{document}